\documentclass[11pt, letterpaper]{article}
\usepackage{amsfonts}
\usepackage{amssymb}
\usepackage{amsmath}
\usepackage{amsthm}
\usepackage{latexsym}

\newtheorem{thm}{Theorem}[section]
\newtheorem{lemma}[thm]{Lemma}
\newtheorem{prop}[thm]{Proposition}
\newtheorem{cor}[thm]{Corollary}
\newtheorem{defi}[thm]{Definition}

\newtheorem*{thm*}{Theorem}

\theoremstyle{definition}
\newtheorem*{remark}{Remark}

\numberwithin{equation}{section}

\newcommand{\lp}{\left(}
\newcommand{\rp}{\right)}

\newcommand{\mR}{\mathbb{R}}

\newcommand{\R}{\mathbb{R}}
\newcommand{\RR}{\mathbb R^{n+1}_+}
\newcommand{\divergence}{\text{div}}
\newcommand{\Divergence}{\text{Div}}

\newcommand{\dvol}{\,dvol}

\newcommand{\halfspace}{\R^{n+1}_+}

\newcommand{\grad}[2]{|\nabla #1|^{#2}}

\newcommand{\abs}[1]{\left\lvert #1\right\rvert}
\newcommand{\norm}[1]{\left\lVert #1\right\rVert}

\newcommand{\be}{\begin{equation}}
\newcommand{\ee}{\end{equation}}
\newcommand{\bee}{\begin{equation*}}
\newcommand{\eee}{\end{equation*}}
\newcommand{\bea}{\begin{eqnarray}}
\newcommand{\eea}{\end{eqnarray}}
\newcommand{\bs}{\begin{split}}
\newcommand{\es}{\end{split}}

\begin{document}

\title{\bf Fractional conformal Laplacians and\\ fractional Yamabe problems}
\author{Mar\'ia del Mar Gonz\'alez \thanks{Gonz\'alez is supported by
Spain Government project MTM2008-06349-C03-01 and GenCat
2009SGR345.}\\ Univ. Polit\`ecnica de Catalunya \and Jie Qing
\thanks{Research of Qing is partially supported by NSF 0700535 and CNSF 10728103.}
\\Univ. of California, Santa Cruz }
\date{}
\maketitle

\abstract{Based on the relations between scattering operators of
asymptotically hyperbolic metrics and Dirichlet-to-Neumann operators
of uniformly degenerate elliptic boundary value problems observed by Chang and Gonz\'alez, we formulate fractional Yamabe problems
that include the boundary Yamabe problem studied by Escobar. We observe an interesting Hopf
type maximum principle together with interplays between analysis
of weighted trace Sobolev inequalities and conformal structure of
the underlying manifolds, which extend the phenomena displayed in
the classic Yamabe problem and boundary Yamabe problem.}


\section{Introduction}

In this paper, based on the relations between scattering operators
of asymptotically hyperbolic metrics and Dirichlet-to-Neumann
operators of uniformly degenerate elliptic boundary value problems
observed in \cite{fractional-laplacian}, we formulated and solved
fractional order Yamabe problems that include the boundary Yamabe problem
studied by Escobar in \cite{Escobar:conformal-deformation}.

Suppose that $X^{n+1}$ is a smooth manifold with smooth boundary
$M^n$ for $n\geq 3$. A function $\rho$ is a defining function of the
boundary $M^n$ in $X^{n+1}$ if
$$
\rho>0 \mbox{ in } X^{n+1}, \quad \rho=0 \mbox{ on } M^n, \quad
d\rho\neq 0 \mbox{ on } M^n.
$$
We say that $g^+$ is conformally compact if, for some defining function $\rho$, the metric $\bar g =
\rho^2g^+$ extends to $\bar X^{n+1}$  so
that $(\bar X^{n+1},\bar g)$ is a compact Riemannian manifold. This
induces a conformal class of metrics $\hat h = \bar g|_{TM^n}$ on
$M^n$ when defining functions vary. The conformal manifold $(M^n,
[\hat h])$ is called the conformal infinity of $(X^{n+1},  g^+)$. A metric $g^+$ is said to be asymptotically hyperbolic if it is
conformally compact and the sectional curvature approaches $-1$ at infinity.

In the recent work \cite{Graham-Zworski:scattering-matrix}, Graham
and Zworski introduced the meromorphic family of scattering
operators $S(s)$, which is a family of pseudo-differential
operators, for a given asymptotically hyperbolic manifold $(X^{n+1},
\ g^+)$ and a choice of the representative $\hat h$ of the conformal
infinity $(M^n, \ [\hat h])$. Often one instead considers the
normalized scattering operators
$$
P_\gamma[g^+,\hat h] = 2^{2\gamma}\frac{\Gamma(\gamma)}{\Gamma(-\gamma)}
S\lp\frac{n}{2}+\gamma\rp.
$$
The normalized scattering operators $P_\gamma [g^+, \hat h]$ are
conformally covariant,
$$
P_\gamma [g^+, w^{\frac 4{n -2\gamma}}\hat h]\phi = w^{-\frac {n +
2\gamma}{n - 2\gamma}}P_\gamma[g^+, \hat h] (w\phi),
$$
with principal symbol
$$
\sigma(P_\gamma[g^+, \hat h]) = \sigma ((-\Delta_{\hat
h})^\gamma).
$$
Hence they may be considered to be conformal fractional Laplacians for
$\gamma \in (0, 1)$ for a given asymptotically hyperbolic metric
$g^+$. As proven in \cite{Graham-Zworski:scattering-matrix},
\cite{Fefferman-Graham:largo}, when $g^+$ is Poincar\'{e}-Einstein,
$P_1$ is the conformal Laplacian, $P_2$ is the Paneitz operator,
and in general $P_k$ for $k\in\mathbb N$ are the conformal powers of the Laplacian
discovered in \cite{Graham-Jenne-Mason-Sparling}.

When $g^+$ is a fixed asymptotically hyperbolic metric we may simply
denote
$$
P_\gamma^{\hat h} := P_\gamma[g^+, \hat h].
$$
We will consider the associated ``fractional order curvature"
$$
Q_\gamma^{\hat h} = P_\gamma^{\hat h}(1),
$$
and the normalized total curvature
$$
I_\gamma[\hat h] = \frac {\int_{M^n} Q_\gamma^{\hat h}dv_{\hat h} }
{ \lp\int_{M^n}dv_{\hat h}\rp^{\frac {n-2\gamma}n} }.
$$
When a background metric $\hat h$ is fixed, we may write
$$
I_\gamma[w, \hat h] = I_\gamma[w^\frac 4{n-2\gamma} \hat h] = \frac
{\int_{M^n} wP_\gamma^{\hat h}wdv_{\hat h}}{\lp\int_{M^n} w^{\frac
{2n}{n-2\gamma}}dv_{\hat h}\rp^{\frac {n-2\gamma}n}}.
$$
This functional $I_\gamma[\hat h]$ is clearly an analogue to the
Yamabe functional. Hence one may ask if there is a metric which is
the minimizer of $I_\gamma$ among metrics in the class $[\hat h]$
and whose curvature $Q_\gamma$ is a constant. We will refer to that
problem as a fractional Yamabe problem when $\gamma\in (0, 1)$. For
the original Yamabe problem readers are refereed to
\cite{Lee-Parker}, \cite{Schoen-Yau:book}. A similar question was
studied in \cite{Qing-Raske:positive-solutions} for $\gamma >1$ and
$g^+$ being a Poincar\'{e}-Einstein metric. Because of the lack of a
maximum principle these generalized Yamabe problems in general are
difficult to solve. Yet this new window to the analytic aspects of
conformal geometry remains fascinating. For example, 
 it was proven in \cite{Guillarmou-Qing} that the location
of the first scattering pole is dictated by the sign of the Yamabe
constant and the Green's function of $P^{\hat h}_\gamma$ is positive for
$\gamma\in (0, 1)$ when the Yamabe constant is positive, at least in
the case where $g^+$ is conformally compact Einstein.

It turns out that one may use the relations of scattering operators
and the Dirichlet-to-Neumann operators to reformulate the above
fractional Yamabe problems as degenerate elliptic boundary value
problems. The correspondence between pseudo-differential equations
and degenerate elliptic boundary value problems is inspired by the
works in \cite{Caffarelli-Silvestre}. Interestingly, the
corresponding degenerate elliptic boundary value problem is a
natural extension of the boundary Yamabe problem raised and studied
in \cite{Escobar:conformal-deformation}.

Recall from \cite{fractional-laplacian} that, given an asymptotically
hyperbolic manifold $(X^{n+1}, \ g^+)$ and a representative $\hat h$
of the conformal infinity $(M^n, \ [\hat h])$, one can find a
geodesic defining function $\rho$ such that the compactified metric can be written as
$$
\bar g := \rho^2g^+ = d\rho^2 + h_\rho = d\rho^2 + \hat h +
h^{(1)}\rho + h^{(2)}\rho^2 + o(\rho^2)
$$
near infinity. One may consider the degenerate elliptic boundary
value problem of $\bar g$ as follows:  \be\label{deg-equ}
\left\{\begin{split}
-\divergence \lp \rho^a \nabla U\rp + E(\rho) U &=0\quad \mbox{in }(X^{n+1}, \ \bar g), \\
U|_{\rho=0}& = f \quad \mbox{on }M^n,
\end{split}\right.\ee where
$$
E(\rho)= \rho^{-1-s} \lp -\Delta_{g^+} - s(n-s)\rp\rho^{n-s},
$$
$s = \frac n2 + \gamma$, and $a = 1 - 2\gamma$. 

\begin{lemma}[Chang and Gonzalez \cite{fractional-laplacian}]
Let $(X^{n+1}, \ g^+)$ be an asymptotically hyperbolic manifold.
Suppose that $U$ is the solution to the boundary value problem
\eqref{deg-equ}. Then
\begin{enumerate}
\item For $\gamma\in(0,\frac{1}{2})$ and $-\frac {n^2}4 + \gamma^2$
 not an $L^2$-eigenvalue for the Laplacian of $g^+$,
\be\label{relation}P_\gamma[g^+, \hat h] f=
-d^*_\gamma\lim_{\rho\to
0}\rho^a
\partial_\rho U,\ee
where
\be\label{multiplicative-constant}d^*_\gamma =-
\frac{2^{2\gamma-1}\Gamma(\gamma)}{\gamma\Gamma(-\gamma)}.\ee
\item For $\gamma=\frac{1}{2}$,
$$P_{\frac{1}{2}}[g^+, \hat h] f= - \lim_{\rho\to 0} \partial_\rho U
+ \tfrac{n-1}{2} H f,$$ where $H := \frac 1{2n}\text{Tr}_{\hat
h}(h^{(1)})$ is the mean curvature of $M$.

\item For $\gamma\in\lp\frac{1}{2},1\rp$, \eqref{relation} still holds
if $H=0$.
\end{enumerate}
\end{lemma}

In light of Lemma 1.1, consider, for $\gamma\in (0,
1)$,
$$
I^*_\gamma [U, \bar g] = \frac {d^*_\gamma \int_{X^{n+1}} (\rho^a
|\nabla U|^2 + E(\rho)U^2)dv_{\bar g}}{\int_{M^n} U^{\frac
{2n}{n-2\gamma}}dv_{\hat h}}.
$$
It is
then a very natural variational problem for $I^*_\gamma$. For
instance, right away one sees that a minimizer of $I^*_\gamma$ is
automatically nonnegative, which was a huge issue for the functional
$I_\gamma$.

One key ingredient in our work here is the following
Hopf type maximum principle. We drew inspiration from some
version of Hopf's lemma  for the Euclidean half space case (Proposition 4.11 in
\cite{Cabre-Sire:I}).

\begin{prop}Let $\gamma\in (0, 1)$. Suppose that $U$ is a nonnegative solution to \eqref{deg-equ}
in $X^{n+1}$. Let $p_0\in M^n= \partial X^{n+1}$ and $B_r$ be a
geodesic ball of radius $r$ centered at $p_0$ in $M^n$. Then, for
sufficiently small $r_0$, if $U(q_0) = 0$ for $q_0 \in B_{r_0}
\setminus \overline{B_{\frac{1}{2}r_0}}$ and $U > 0$ on $\partial
B_{\frac 1 2 r_0}$, then
\begin{equation}
y^a \partial_y U|_{q_0} > 0.
\end{equation}
\end{prop}

\noindent It seems weaker than the original one, but it suffices for
our purposes. A nice and immediate consequence of the above maximum
principle is that the first eigenfunction of the fractional
conformal Laplacian $P_\gamma^{\hat h}$ is always positive, which
has been a rather challenging question in general for the
pseudo-differential operators $P_\gamma^{\hat h}$ (cf.
\cite{Guillarmou-Qing}). Hence one can produce a metric in the class
$[\hat h]$ that has positive, negative, or zero $Q_\gamma$ curvature
when the first eigenvalue is positive, negative, or zero
respectively.

Our approach to solve the $\gamma$-Yamabe problem is very similar to
the one taken in \cite{Escobar:conformal-deformation}, where one of the crucial steps is the understanding of a trace inequality. In our case, the relevant
sharp weighted trace Sobolev inequality appeared in the works
\cite{Lieb:sharp-constants},
\cite{Cotsiolis-Tavoularis:best-constants},
\cite{Nekvinda:characterization-tracesW}:

\begin{prop}
Let $\gamma \in (0, 1)$ and $a = 1 - 2\gamma$. Suppose that $U\in
W^{1,2}(\RR,y^a)$ with trace $TU=w$. Then, for some constant $\bar
S(n,\gamma)$, \be \norm{w}_{L^{2^*}(\R^n)}^2\leq \bar S(n,\gamma)
\int_{\RR} y^a \grad U 2\,dxdy,\ee where $2^* = \frac
{2n}{n-2\gamma}$. Moreover the equality holds if and only if
$$
w(x) = c \lp \frac{\mu}
{\abs{x-x_0}^2+\mu^2}\rp^{\frac{n-2\gamma}{2}},\quad x\in \R^n,
$$
for $c\in\mathbb R$, $\mu>0$ and $x_0\in\R^n$ fixed, and $U$
is its Poisson extension of $w$ as given in \eqref{Poisson}.
\end{prop}

As in the case of original Yamabe problem, one can define the
$\gamma$-Yamabe constant
$$
\Lambda_\gamma(M^n, [\hat h]) = \inf_{h \in [\hat h]}I_\gamma[h].
$$
It is then easily seen that
$$
\Lambda_\gamma(S^n, [g_c]) = \frac{d^*_\gamma}{\bar S(n, \gamma)}
$$
where $[g_c]$ is the canonical conformal class of metrics on the
sphere $S^n$. Analogous to the cases of the original Yamabe problem
we obtain

\begin{thm}\label{thm1} Suppose that $(X^{n+1}, g^+)$ is an
asymptotically hyperbolic manifold. Suppose, in addition,  that $H=0$ when
$\gamma\in (\frac 12, 1)$. Then, if \be\label{condition}
-\infty < \Lambda_\gamma(M,[\hat h])<\Lambda_\gamma(S^n,[g_c]),\ee then the
$\gamma$-Yamabe problem is solvable for $\gamma \in (0, 1)$.
\end{thm}

\begin{remark} It is easily seen that $\Lambda_\gamma(M, [\hat h]) > -\infty$ in the light of
(1.4) in Theorem 1.1 and Theorem 1.2  in \cite{Jin-Jing} when $\gamma \in (0, \frac 12]$ or if
some additional assumptions in Theorem 1.2 in \cite{Jin-Jing} hold.
\end{remark}

Based on computations similar to ones in
\cite{Escobar:conformal-deformation}, we have

\begin{thm}\label{thm2} Suppose that $(X^{n+1}, \ g^+)$ is an
asymptotically hyperbolic manifold and that
\be\label{cur-condi} \rho^{-2}\big(R[g^+] - Ric[g^+](\rho\partial_\rho)
+ n^2\big) \to 0 \quad \text{as $\rho\to 0$.}\ee If $X^{n+1}$ has a
non-umbilic point on $\partial X^{n+1}$ and
\be\label{cst}-\frac{n+a-3}{1-a}2^{2\gamma+1}
\frac{\Gamma(\gamma)}{\Gamma(-\gamma)}+\frac{n-1+a}{a+1}<0,\ee then
$$\Lambda_\gamma(M,[\hat h])<\Lambda_\gamma(S^n,[g_c])$$
and hence the $\gamma$-Yamabe problem is solvable for $\gamma \in
(0, 1)$.
\end{thm}

We remark now that the $\frac
12$-Yamabe problem introduced in here reduces back to the boundary
Yamabe problem consider in \cite{Escobar:conformal-deformation} in
this way. Notice that, in this case, we have
\be\label{boundary-Yamabe-energy} I^*_\frac 12 [U, \ \phi^\frac
4{n-1}\bar g] = I^*_\frac 12[U\phi, \ \bar g]\ee for any positive
function $\phi$ on $\bar X^{n+1}$ and therefore \eqref{cur-condi} is
no longer needed. Also notice that the condition \eqref{cst} becomes
$n>5$ when $\gamma = \frac 12$, which agrees with the conclusion in
\cite{Escobar:conformal-deformation}.\\

Suppose we start with a compact Riemannian manifold $(X^{n+1}, \
\bar g)$ and its boundary $(M^n, \ \hat h)$. Then one can construct
an asymptotically hyperbolic manifold $(X^{n+1}, \ g^+)$ which is
conformal to $(X^{n+1}, \ \bar g)$. For example, as observed in
\cite{fractional-laplacian}, one may require according to the works
in \cite{Mazzeo:regularity-singular-Yamabe},
\cite{Andersson-Chrusciel-Friedrich} that
\be\label{const-scalar-curv} R[g^+] = -n (n+1). \ee
Then the induced
degenerate equation becomes \be\label{Escobar-div} -\divergence \lp
\rho^a \nabla U\rp + \tfrac{n-1+a}{4n} R[\bar g] \rho^a U =0\quad
\mbox{in }(X^{n+1},\bar g) \ee  whose  associated variational functional
becomes \be\label{Escobr-energy}F[U] = \int_X \rho^a |\nabla
U|_{\bar g}^2\,dv_{\bar g} + \tfrac{n-1+a}{4n} \int_X R[\bar g]
\rho^a|U|^2\,dv_{\bar g}.\ee

In section 2 we recall the work from \cite{fractional-laplacian} to make
possible the passage from pseudo-differential equations to second
order elliptic boundary value problems as in
\cite{Caffarelli-Silvestre}. In Section 3 we study regularity ($L^\infty$ and Schauder estimates) for
degenerate elliptic boundary value problems. And more importantly we
 establish the Hopf type maximum principle. In Section 4 we
formulate the fractional Yamabe problem and obtain some properties
for the fractional case that are analogous to the original Yamabe
problem with the help of the Hopf-type maximum principle. In Section
5 we analyze sharp weighted Sobolev trace
inequalities. We define, on any conformal manifold, the
fractional Yamabe constant associated with an asymptotically
hyperbolic metric and show that the one of the standard round spheres
associated to the standard hyperbolic metric is the largest. In Section
6 we take a subcritical approximation and prove our Theorem
\ref{thm1}. In the last section we adopt the calculation from
\cite{Escobar:conformal-deformation} and prove our Theorem
\ref{thm2} by choosing a suitable test function.

We finally mention the two related works \cite{Barrios-Colorado-DePablo-Sanchez,Servadei} on nonlinearities with critical exponents for the fractional Laplacian.


\section{Conformal fractional Laplacians}

In this section we introduce the recent works in
\cite{fractional-laplacian} to relate two equivalent definitions of
conformal fractional Laplacians. Conformal fractional Laplacians are
defined via scattering theory on asymptotically hyperbolic manifolds
in \cite{Graham-Zworski:scattering-matrix},
\cite{Fefferman-Graham:largo}. We also have seen fractional
Laplacians defined as Dirichlet-to-Neumann operators for degenerate
equations on compact manifolds with boundary in
\cite{Caffarelli-Silvestre}. It turns out in some way these two
fractional Laplacians are the same.

Let $X^{n+1}$ be a smooth manifold of dimension $n+1$ with compact boundary
$\partial X = M^n$. A function $\rho$ is a \emph{defining function}
of $\partial X$ in $X$ if
$$
\rho>0 \mbox{ in } X, \quad \rho=0 \mbox{ on }\partial X, \quad
d\rho\neq 0 \mbox{ on } \partial X.
$$
We say that $g^+$ is \emph{conformally compact} if the metric $\bar
g = \rho^2g^+$ extends to $\bar X^{n+1}$ for a defining function $\rho$ so
that $(\bar X^{n+1},\bar g)$ is a compact Riemannian manifold. This
induces a conformal class of metrics $\hat h = \bar g|_{TM^n}$ on
$M^n$ when the defining function varies, which is called the
\emph{conformal infinity} of $(X^{n+1},  g^+)$. A metric $g^+$ is
said to be \emph{asymptotically hyperbolic} if it is conformally
compact and the sectional curvature approaches to $-1$ at infinity.

Given an asymptotically hyperbolic manifold $(X^{n+1}, g^+)$ and a
representative $\hat h$ of the conformal infinity $(M^n, [\hat h])$,
there is a uniquely geodesic defining function $\rho$ such that, on
a neighborhood $M \times (0,\delta)$ in $X$, $g^+$ has the normal form
\be\label{normal-form}g^+ = \rho^{-2}(d\rho^2 + h_\rho)\ee where
$h_\rho$ is a one parameter family of metrics on $M$ such that
\be\label{general-expansion} h_\rho = \hat h + h^{(1)}\rho +
O(\rho^2) .\ee


From \cite{Mazzeo-Melrose:meromorphic-extension},
\cite{Graham-Zworski:scattering-matrix} it follows that, given
$f\in\mathcal C^\infty(M)$, $Re(s) > \frac n2$ and $s(n-s)$ is not a
$\rm{L}^2$-eigenvalue for $-\Delta_{g^+}$, the generalized
eigenvalue problem \be\label{equation-GZ}
-\Delta_{g^+}u-s(n-s)u=0,\quad\mbox{in } X \ee has a solution of the
form \be\label{general-solution} u = F \rho ^{n-s} + G\rho^s,\quad
F,G\in\mathcal C^\infty(\bar{X}),\quad F|_{\rho=0}=f. \ee The
\emph{scattering operator} on $M$ is then defined as
$$
S(s)f = G|_M.
$$
It is shown in \cite{Graham-Zworski:scattering-matrix} that, by
a meromorphic continuation, $S(s)$ is a meromorphic family of
pseudo-differential operators in the whole complex plane. Instead, it is often useful to consider the normalized scattering operators
$P_\gamma[g^+,\hat h]$ defined as:
\be\label{P-operator} P_\gamma[g^+,\hat h] := d_\gamma
S\lp\frac{n}{2}+\gamma\rp,\quad
d_\gamma=2^{2\gamma}\frac{\Gamma(\gamma)}{\Gamma(-\gamma)}.\ee
Note
that $s = \frac n2 + \gamma$. With
this regularization the principal symbol of $P_\gamma[g^*,\hat h]$ is exactly
the principal symbol of the fractional Laplacian $(-\Delta_{\hat
h})^{\gamma}$. Hence we will call (assuming implicitly the dependence on the extension metric $g^+$)
$$
P_\gamma^{\hat h} := P_\gamma [g^+, \hat h]
$$
a conformal fractional Laplacian for each $\gamma\in (0, 1)$ which
is not a pole of the scattering operator, i.e. $\frac {n^2}4 -
\gamma^2$ is not a $\rm{L}^2$-eigenvalue for $-\Delta_{g^+}$.  It is a conformally covariant
operator, in the sense that it behaves like
\be\label{conformal-invariance}P_\gamma^{\hat h_w}\varphi=
w^{-\frac{n+2\gamma}{n-2\gamma}} P_\gamma^{\hat h} (w\varphi)\ee for
a conformal change of metric $\hat h_w=w^{\frac{4}{n-2\gamma}}\hat
h$. We will call
$$
Q_\gamma^{\hat h} = P_\gamma^{\hat h}(1)
$$
the fractional scalar curvature associated to the conformal
fractional Laplacian $P_\gamma^{\hat h}$. From the above
\eqref{conformal-invariance} we have \be\label{fractional PDE}
P_{\gamma}^{\hat h}(w)= Q_\gamma^{\hat
h_w}w^{\frac{n+2\gamma}{n-2\gamma}}. \ee

The familiar case is
$\gamma = 1$, where
$$
P_1^{\hat h} = -\Delta_{\hat h} + \frac {n-2}{4(n-1)}R[\hat h]
$$
becomes the conformal Laplacian and the associated curvature is the
scalar curvature $Q_1^{\hat h} = \frac {n-2}{4(n-1)} R[\hat h]$ of
the metric $\hat h$ which undergoes the change
$$P_1^{\hat h}w = \frac {n-2}{4(n-1)} R[\hat h_w]w^\frac {n+2}{n-2}$$
when taking conformal change of metrics, provided that $(X^{n+1}, \
g^+)$ is a Poincar\'{e}-Einstein as established in
\cite{Graham-Zworski:scattering-matrix},
\cite{Fefferman-Graham:largo}. The conformal fractional Laplacians
and fractional scalar curvatures should also be compared to the
higher order generalization of the conformal Laplacian and scalar
curvature: the Paneitz operator $P_2^{\hat h}$ and its associated
$Q$-curvature (see \cite{Paneitz:published},
\cite{Branson:sharp-inequalities},
\cite{Qing-Raske:positive-solutions}).\\

It was observed by Chang and Gonz\'alez in
\cite{fractional-laplacian} that the generalized eigenvalue problem
\eqref{equation-GZ} on a non-compact manifold $(X^{n+1}, g^+)$ is
equivalent to a linear degenerate elliptic problem on the compact
manifold $(\bar X^{n+1}, \bar g)$, for $\bar g=\rho^2 g^+$. Hence
Chang and Gonz\'alez reconciled the definition of the fractional
Laplacians given in the above as normalized scattering operators and
the one given in the spirit of the Dirichlet-to-Neumann operators by
Caffarelli and Silvestre in \cite{Caffarelli-Silvestre}. This
observation in \cite{fractional-laplacian} plays a fundamental role
in this paper and provides an alternative way to study the
fractional partial differential equation \eqref{fractional PDE}.
First, we know by the conformal covariance that
$$
P_1^{g^+} u = \rho^{\frac {n+3}2} P_1^{\bar g} (\rho^{-\frac {n-1}2}
u).
$$
Let $a = 1 - 2\gamma\in (-1, 1)$, $s = \frac n2 + \gamma$, and $U =
\rho^{s-n} u$. Then we may write the equation \eqref{equation-GZ} as
$$
-\divergence(\rho^a \nabla_{\bar g} U) + E(\rho) U = 0,\quad
\mbox{in }(X^{n+1}, \ \bar g),
$$
where \be\label{E2}E(\rho):=\rho^\frac a2 P_1^{\bar g}
\rho^{\frac{a}{2}} - \lp s(n-s) + \tfrac{n-1}{4n}R[{g^+}]\rp
\rho^{a-2},\ee or writing everything back in the metric $g^+$,
\be\label{E1} E(\rho)= \rho^{-1-s} \lp -\Delta_{g^+} -
s(n-s)\rp\rho^{n-s}.\ee Notice that, in a neighborhood
$M\times(0,\delta)$ where the metric $g^+$ is in the normal form,
\be\label{E-normal-form} E(\rho)=\tfrac{n-1+a}{4n} \left[ R[{\bar g}]
- (n(n+1) + R[{g^+}])\rho^{-2} \right] \rho^a\quad \mbox{in } M\times
(0,\delta). \ee 

\begin{prop}[Chang and Gonz\'alez  \cite{fractional-laplacian}]\label{prop-Chang-Gonzalez}
Let  $(X^{n+1}, \ g^+)$  be an asymptotically hyperbolic manifold.
Then, given $f\in\mathcal C^{\infty}(M)$, the generalized eigenvalue
problem \eqref{equation-GZ}-\eqref{general-solution} is equivalent
to \be\label{div}\left\{\begin{split}
-\divergence \lp \rho^a \nabla U\rp + E(\rho) U &=0\quad \mbox{in }(X,\bar g), \\
U|_{\rho=0}&=f\quad \mbox{on }M,
\end{split}\right.\ee
where $U = \rho^{n-s} u$ and $U$ is the unique minimizer of the
energy
$$F[V] = \int_X \rho^a |\nabla V|_{\bar g}^2\,dv_{\bar g} + \int_X E(\rho)|V|^2\,dv_{\bar g}$$
among all the functions $V\in W^{1,2}(X,\rho^a)$ with fixed trace
$V|_{\rho=0}=f$. Moreover,
\begin{enumerate}
\item For $\gamma\in(0,\frac{1}{2})$,
\be\label{compute-fractional}P_\gamma^{\hat h} f=
-d^*_\gamma\lim_{\rho\to 0}\rho^a \partial_\rho U,\ee
where the constant $d^*_\gamma$ is given in \eqref{multiplicative-constant}.

\item For $\gamma=\frac{1}{2}$, we have an extra term
$$P_{\frac{1}{2}}^{\hat h} f= - \lim_{\rho\to 0} \partial_\rho U
+ \tfrac{n-1}{2} H f,$$ where $H := \frac 1{2n}\text{Tr}_{\hat
h}(h^{(1)})$ is the mean curvature of $M$.

\item For $\gamma\in\lp\frac{1}{2},1\rp$, \eqref{compute-fractional} still holds
if and only if $H=0$.
\end{enumerate}
\end{prop}

\begin{remark} It should be noted here that there are many asymptotically
hyperbolic manifolds $(X^{n+1}, \ g^+)$ whose conformal infinity is
prescribed as $(M^n, \ [\hat h])$. If one insists $(X^{n+1}, \ g^+)$
to be Poincar\'{e}-Einstein, then the normalized scattering
operators $P_\gamma^{\hat h}$ are a bit more intrinsic, at least at
positive integers as observed in
\cite{Graham-Zworski:scattering-matrix},
\cite{Fefferman-Graham:largo}. It should also be noted that one can
simply start with a compact Riemannian manifold $(\bar X^{n+1}, \
\bar g)$ with boundary $(M^n, \ \hat h)$ and easily build an
asymptotically hyperbolic manifold whose conformal infinity is given
by $(M^n, \ [\hat h])$. Please see the details of this observation
in \cite{fractional-laplacian}.
\end{remark}

The simplest example of a conformally compact Einstein manifold is the
hyperbolic space $(\mathbb H^{n+1}, g_{\mathbb H})$. It can be
characterized as the upper half-space (with coordinates $x\in\mathbb
R^n$, $y\in \mathbb R_+$), endowed with the  metric:
$$g^+=\frac{dy^2+|dx|^2}{y^2}.$$
Then \eqref{div} with Dirichlet condition $w$ reduces to
\bee\left\{\begin{split}
-\divergence \lp y^a \nabla U\rp &=0 \quad\mbox{in }\mathbb R^{n+1}_+,\\
U|_{y=0}&=w \quad\mbox{on }\mathbb R^n,
\end{split}\right.\eee
and the fractional Laplacian at the boundary $\mathbb R^n$ is just
$$
P_\gamma^{|dx|^2} w = (-\Delta_{|dx|^2})^\gamma w = -
d_\gamma^*\,\lim_{y\to 0} \lp y^a \partial_y U\rp.
$$
This is precisely the Caffarelli-Silvestre extension
\cite{Caffarelli-Silvestre}. Note that this extension $U$ can be
written in terms of the Poisson kernel $K_\gamma$ as follows:
\be\label{Poisson}U(x,y)=K_\gamma *_x w=C_{n,\gamma}\int_{\mathbb
R^n} \frac{y^{1-a}}{( \abs{x-\xi}^2+\abs{y}^2)^{\frac{n+1-a}{2}}}
\;w(\xi)\,d\xi, \ee for some constant $C_{n,\gamma}$. Moreover,
given $w\in H^\gamma(\mathbb R^n)$, $U$ is the minimizer of the functional:
$$
F[V]=\int_{\mathbb R^{n+1}_+} y^a\grad V 2 \,dxdy
$$
among all the possible extensions in the set
$$
\left\{V:\mathbb R^{n+1}_+\to \mathbb R \,:\, \int_{\mathbb
R^{n+1}_+} y^a\grad V 2 \,dxdy <\infty, \;V(\cdot,0)=w\right\}.
$$

Based on \eqref{E1} it is observed in \cite{fractional-laplacian}
that one may use
$$
\rho^* = v^\frac 1{n-s}
$$
as a defining function, where $v$ solves
$$
-\Delta_{g^+} v - s(n-s) v = 0
$$
and $\rho^{s-n}v = 1$ on $M$, to eliminate $E(\rho^*)$ from equation
\eqref{div}. It suffices to show that $v$ is strictly positive in
the interior. But this is true because, away from the boundary, it
is the solution of an uniformly elliptic equation in divergence
form, thus it cannot have a non-positive minimum. Hence we arrive at
an improvement of Proposition \ref{prop-Chang-Gonzalez} as follows:

\begin{prop}\label{new-defining-function}
The function $\rho^*$ is a defining function of $M$ in $X$ such that
$E(\rho^*)\equiv 0$. Hence $U = (\rho^*)^{s-n}u$ solves
\be\label{new-extension}\left\{
\begin{split}
-\divergence \lp (\rho^*)^a \nabla U\rp&=0 \quad \mbox{in }(X, \bar g^*),\\
U&=w \quad\mbox{on }M,
\end{split}\right.\ee
with respect to the metric $\bar g^*=(\rho^*)^2 g^+$ and $U$ is the
unique minimizer of the energy \be\label{functional} F[V] = \int_X
(\rho^*)^a |\nabla V|_{\bar g^*}^2 \,dv_{\bar g^*}\ee among all the
extensions $V\in W^{1,2}(X,(\rho^*)^a)$ satisfying $V|_M=w$.
Moreover,
$$\rho^*(\rho)=\rho\left[ 1+\frac{Q^{\hat h}_\gamma}{(n-s)d_\gamma}  \rho^{2\gamma}+O(\rho^2)\right]$$
near the infinity and \be\label{property10} P_\gamma^{\hat h} w=
-d^*_\gamma \lim_{{\rho^*}\to 0} (\rho^*)^a\partial_{\rho^*} U+w
Q_\gamma^{\hat h},\ee provided that $H=0$ when $\gamma\in (\frac 12,
1)$.
\end{prop}

We will sometimes use the defining function $\rho^*$, denoted by $y$ unless
explicitly stated otherwise, because it allows us to work with a pure
divergence equation with no lower order terms.\\

We end this section by discussing the assumption that $H=0$ for
an asymptotically hyperbolic metric $g^+$. It turns out that this
indeed is an intrinsic condition.

\begin{lemma}\label{mean-curvature} Suppose that $(X^{n+1}, \ g^+)$ is an asymptotically
hyperbolic manifold and that $\rho$ and $\tilde \rho$ are the
geodesic defining functions of $M$ in $X$ associated with
representatives $\hat h$ and $\tilde h$ of the conformal infinity
$(M^n, \ [\hat h])$ respectively. Hence
$$
g^+ = \rho^{-2}(d\rho^2 + h_\rho) = \tilde\rho^{-2}(d\tilde\rho^2 +
\tilde h_{\tilde\rho})
$$
where
$$
h_\rho = \hat h + \rho h^{(1)} + O(\rho^2)
$$
and
$$
\tilde h_{\tilde\rho} = \tilde h + \tilde\rho \tilde h^{(1)} +
O(\tilde\rho^2)
$$
near the infinity. Then
$$
\tilde h^{(1)} =  h^{(1)} \text{\quad on $M$}.
$$
In particular
$$
H  = \left.\frac {\tilde\rho}{\rho}\right|_{\rho = 0} \tilde H  \text{\quad on
$M$}.
$$
\end{lemma}

\begin{proof} This simply follows from the equations that define the
geodesic defining functions. Let
$$
\tilde \rho = e^w \rho
$$
near the infinity. Then
$$
1 = |d(e^w\rho)|^2_{e^{2w}\rho^2g^+} = |d\rho|^2_{\rho^2 g^+} + 2
\rho \langle dw, d\rho\rangle_{\rho^2 g^+} + \rho^2 |dw|^2_{\rho^2 g^+},
$$
which implies
$$
2 \frac {\partial w}{\partial \rho} + \rho \left[ \lp\frac {\partial
w}{\partial \rho}\rp^2 + |\nabla w|^2_{h_\rho}\right] = 0.
$$
Hence it is rather obvious that $\frac {\partial w}{\partial
\rho} = 0$ at $\rho = 0$. Therefore the proof is complete in the
light of the fact that
$$
\tilde g = \tilde\rho^2 g^+ = e^{2w}\rho^2 g^+ = e^{2w}\bar g.
$$
\end{proof}


\section{Uniformly degenerate elliptic equations} \label{section-elliptic}

Considering the fractional powers of the
Laplacian as Dirichlet-to-Neumann operators in Proposition
\ref{new-defining-function} allows to relate the properties of
non-local operators to those of uniformly degenerate elliptic
equations in one more dimension. The same strategy has been used,
for instance, in the recent work of Cabr\'e-Sire
\cite{Cabre-Sire:I}.

Fix $\gamma\in(0,1)$. Let $y=\rho^*$ be the special defining
function given in Proposition \ref{new-defining-function} and set
$\bar g^*=y^2 g^+$. We are concerned with the uniformly degenerate
elliptic equation \be\label{degenerate} \left\{
\begin{split}
-\divergence \lp y^a \nabla U\rp&=0 \quad \mbox{in }(X, \bar g^*),\\
U & = w \quad\mbox{on }M.
\end{split}\right.\ee
For our purpose we will concentrate on the local behaviors of the
solutions to \eqref{degenerate} near the boundary. First, we write
our equation in local coordinates near a fixed boundary point
$(p_0,0)$. More precisely, for some $R>0$, we set \bee\begin{split}
B_R^+ &=\{(x,y)\in\mathbb R^{n+1}: y>0, \abs{(x,y)}<R\}, \\
\Gamma_R^0 &=\{(x,0)\in\partial \mathbb R^{n+1}_+: \abs{x}<R\}, \\
\Gamma_R^+ &=\{(x,y)\in\mathbb R^{n+1}: y\geq0, \abs{(x,y)}=R\}.
\end{split}\eee
In local coordinates on $\Gamma_R^0$ the metric $\hat h$ is of the
form $\abs{dx}^2 (1+O(\abs{x}^2))$, where $x(p_0)=0$. Consider the
matrix
$$
A(x,y)=\sqrt{\abs{\det\bar g^*}}y^a (\bar g^*)^{-1}.
$$
Then the equation \eqref{degenerate} is equivalent to
\be\label{FKS}\sum_{i,j=1}^{n+1}\partial_i \lp A_{ij}\partial_j U\rp
=0.\ee Moreover we know that \be\label{ellipticity}\frac{1}{c}y^a
I\leq A\leq c y^a I.\ee This shows that  \eqref{FKS} is a uniformly
degenerate elliptic equation. For instance, the weight $\psi(y)=y^a$
is an $\mathcal A_2$ weight in the sense of \cite{Muckenhoupt:Hardy}. Equation
\eqref{FKS} has been well understood in a series of papers by Fabes,
Jerison, Kenig, Serapioni
(\cite{Fabes-Kenig-Serapioni:local-regularity-degenerate},
\cite{Fabes-Jerison-Kenig:Wiener-test}). Let us state a regularity
result that is relevant to us. We will concentrate on problems of
the form \be\label{eq-general}\left\{\begin{split}
\Divergence(A (DU))&=0 \quad\mbox{in } B_R^+,\\
-y^a\partial_y U&=F,\quad\mbox{on } \Gamma_R^0,
\end{split}\right.\ee
where, for the rest of the section, $A$ satisfies the ellipticity
condition \eqref{ellipticity} for $a\in(-1,1)$, the derivatives are
Euclidean, that is, $D:=\lp\partial_{x_1},\ldots,\partial_{x_n},y\rp$, and
$$\Divergence(A (DU)):=\sum_{i,j=1}^{n+1}\partial_i \lp A_{ij}\partial_j U\rp.$$

\begin{defi}
Given  $R>0$ and a function $F\in L^1(\Gamma_R^0)$, we call $U$ a
weak solution of \eqref{eq-general} if $U$ satisfies
$$(D U)^t A (D U)\in L^1(B_R^+)$$
and
$$\int_{B_R^+} (D \phi)^t A (D U) \,dxdy-\int_{\Gamma_R^0} F\phi\,dx=0$$
for all $\phi\in \mathcal C^1(\overline{B_R^+})$ such that $\phi\equiv
0$ on $\Gamma_R^+$ and $(D \xi)^t A (D \phi)\in L^1(B_R^+).$
\end{defi}

H\"older regularity for weak solutions was shown in
\cite{Fabes-Kenig-Serapioni:local-regularity-degenerate}, Lemma
2.3.12, for any $A$ satisfying \eqref{ellipticity}.  Using this main
result, regularity of weak solutions up to the boundary was
carefully shown in \cite{Cabre-Sire:I}, Lemma 4.3, at least when
$A=y^a I$. However, their proof only depends on the divergence
structure of the equation and the behavior of the weight. Hence we
have

\begin{prop}\label{prop-regularity} Let $\gamma\in(0,1)$, $\gamma=\frac{1-a}{2}$ and $\beta\in(0,\min\{1,1-a\})$.
Let $R>0$ and $U\in L^\infty(B_{2R^+})\cap W^{1,2}(B_{2R}^+,y^a)$ be
a weak solution of
\be\label{degenerate-equation}\left\{\begin{split}
\Divergence(A (DU))&=0 \quad\mbox{in } B_{2R}^+,\\
-y^a\partial_y U&=F(U)\quad\mbox{on } \Gamma_{2R}^0,
\end{split}\right.\ee
for $A$ satisfying \eqref{ellipticity}. If $F\in\mathcal
C^{1,\beta}$, then $U\in \mathcal C^{0,\tilde
\beta}(\overline{B_R^+})$  and $\partial_{x_i} U\in \mathcal
C^{0,\tilde\beta}(\overline{B_R^+})$, $i=1,\ldots,n$, for some
$\tilde\beta\in(0,1)$.
\end{prop}

Particularly, when $F(x,t) = \alpha (x) t + \beta(x) t^\frac
{n+2\gamma}{n-2\gamma}$, to get smoothness it is necessary to
know the local boundedness of weak solutions $U$ on
$\overline{B_R^+}$. To get this local boundedness for weak solutions
we employ the usual Moser's iteration scheme adapted to boundary
valued problems (see Theorem \ref{thm-moser} below).  However, a new
idea is required:  we will perform two coupled iterations, one in the
interior and one at the boundary, that need to be handled
simultaneously. Note that in the linear case when $F\equiv 0$, local
boundedness was shown in {\cite[Corollary 2.3.4]
{Fabes-Kenig-Serapioni:local-regularity-degenerate}}, using the weighted
Sobolev embeddings in the interior described in Proposition
\ref{prop-weighted-Sobolev}. However, when a non-linearity $F(U)$ is
present at the boundary term, instead we need to use  weighted trace
Sobolev embeddings.

First, we recall a weighted Sobolev embedding theorem in the
interior (c.f. \cite[Theorem
1.3]{Fabes-Kenig-Serapioni:local-regularity-degenerate}, see also \cite{Chiarenza-Frasca}):

\begin{prop}\label{prop-weighted-Sobolev}
Let $\Omega$ be an open bounded set in $\mR^{n+1}$. Take
$1<p<\infty$. There exist positive constants $C_\Omega$ and $\delta$
such that for all $u\in \mathcal C_0^\infty(\Omega)$ and all $k$
satisfying $1\leq k\leq \frac{n+1}{n}+\delta$,
$$\norm{u}_{L^{kp}(\Omega,y^a)}\leq C_\Omega \norm{\nabla  u}_{L^p (\Omega,y^a)}.$$
$C_\Omega$ maybe taken to depend only on $n$, $p$, $a$ and the
diameter of $\Omega$.
\end{prop}

Now we can state the theorem. Note that we actually prove it in the
flat case but it is straightforward to generalize it to the manifold
setting:

\begin{thm}\label{thm-moser}
Let $U$ be a weak solution of the problem
\be\label{equation-moser}\left\{\begin{split}
\divergence(y^a\nabla U)&=0 \quad\mbox{in } B_{2R}^+,\\
-y^a\partial_y U&=F(U)\quad\mbox{on } \Gamma_{2R}^0,
\end{split}\right.\ee
where $F(z)$ satisfies
$$F(z)=O\lp\abs{z}^{\beta-1}\rp, \quad \mbox{when }\abs{z}\to\infty,\quad \mbox{for some}\quad 2<\beta<2^*.$$
Assume, in addition, that $\int_{\Gamma_{2r_0}^0}
|U|^{2^*}\,dx=:V<\infty$. Then for each $\bar p>1$, there exists a
constant $C_{\bar p}=C(\bar p,V)>0$ such that
$$ \sup_{B_R^+} |U| +\sup_{\Gamma_R^0} |U|\leq C_{\bar p} \left[\lp\frac{1}{R^{n+1+a}}\rp^{1/\bar p}\norm{U}_{L^{\bar p}(B_{2R},y^a)}+ \lp\frac{1}{R^n}\rp^{1/\bar p}\norm{U}_{L^{\bar p}(\Gamma_{2R}^0)}\right].$$
\end{thm}

\begin{proof} Let $p\in\partial X$.
Note that we can work with normal coordinates
$x_1,\ldots,x_n\in\mR^n$, $y>0$ near $p$. Without loss of
generality, assume that $R=1$. Then the general case is obtained by
rescaling. Let $\eta=\eta(r)$, $r=(\abs{x}^2+y^2)^{1/2}$, be a smooth
cutoff function such that $\eta=1$ if $r<1$, $\eta=0$ if $r\geq 2$,
$0\leq \eta\leq 1$ if $r\in(1,2)$. Next, by working with
$U^+:=\max\{U,0\}$, $U^-:=\max\{-U,0\}$ separately, we can assume
that $U$ is positive.

A good reference for Moser iteration arguments in divergence
structure equations is \cite[chapter 8]{Gilbarg-Trudinger}. We
generalize this method, considering a double iteration: one at the
boundary, using Sobolev trace inequalities to handle the non-linear
term $F(U)$, the other in the interior domain.

The first step is to use that $U$ is a weak solution of
\eqref{equation-moser} by finding a good test function. Formally we
can write the following: multiply equation \eqref{equation-moser} by
$\eta^2 U^\alpha$ and integrate by parts: \be\label{formula46} 0
=2\int_{B_{2}^+} y^a\eta U^\alpha \nabla \eta\nabla U\,dxdy+\alpha
\int_{B_{2}^+} y^a\eta^2 U^{\alpha-1}\grad U
2\,dxdy+\int_{\Gamma_{2}^0} \eta^2 U^\alpha F(U)\,dx. \ee
This implies, using H\"older estimates to handle the crossed term,
\be\label{formula42} \int_{B_{2}^+} y^a\eta^2
U^{\alpha-1}\grad U 2\,dxdy \leq \frac 2\alpha\int_{\Gamma_{2}^0}
\eta^2 U^\alpha F(U)\,dx + \frac 4{\alpha^2}\int_{B_{2}^+} y^a
|\nabla \eta|^2 U^{\alpha+1}\,dxdy.\ee
On the other hand, again using  H\"{o}lder inequality, we have \bee
\int_{B_{2}^+} y^a|\nabla (\eta U^\delta)|^2\,dxdy \leq 2
\delta^2\int_{B_{2}^+} y^a \eta^2  U^{2(\delta-1)}|\nabla U|^2\,dxdy
+ 2\int_{B_{2}^+} y^a U^{2\delta} \abs{\nabla \eta}^2\,dxdy. \eee If
we insert formula \eqref{formula42} into the inequality above, for
the choice $\alpha=2\delta-1$, we obtain
\be\label{formula-moser}\begin{split} J & :=\int_{B_{2}^+}
y^a|\nabla (\eta U^\delta)|^2\,dxdy \\ & \leq 2\lp 1 + \lp\tfrac
{\alpha+1}\alpha\rp^2\rp \int_{B_{2}^+} y^a |\nabla \eta|^2
U^{2\delta}\,dxdy + \frac {(\alpha+1)^2}\alpha \int_{\Gamma_{2}^0}
\eta^2 U^\alpha F(U)\,dx \\ & =:I_1+I_2.\end{split} \ee
For the left hand side above, recall the trace Sobolev embedding
(Corollary \ref{cor-trace-inequality}):
\be\label{formula-trace}J=\int_{B_{2}^+} y^a|\nabla (\eta
U^\delta)|^2\,dxdy\gtrsim \lp\int_{\Gamma_{2}^0} (\eta
U^\delta)^{2^*} \,dx\rp^{\frac{2}{2^*}},\ee and the standard
weighted Sobolev embedding from Proposition
\ref{prop-weighted-Sobolev}. \be\label{formula-Sobolev}J =
\int_{B_{2}^+} y^a|\nabla (\eta U^\delta)|^2\,dxdy \gtrsim \lp
\int_{B_2^+} y^a(\eta U^\delta)^{k}\rp^{\frac{2}{ k}}\ee for some
$1<k<2\frac{n+1}{n}.$

Next, we estimate from above the terms $I_1,I_2$ in \eqref{formula-moser}.
$I_1$ can be easily handled since $|\nabla \eta|\leq C$:
\be\label{formulaI1}I_1=\int_{B_{2}^+}
y^a |\nabla \eta|^2 U^{2\delta}\,dxdy\lesssim \int_{B_2^+}
y^aU^{2\delta}\,dxdy.\ee
Now we consider the second term. To estimate $I_2$, if we write
$U^{2\delta-2+\beta}=U^{\beta-2} U^{2\delta}$, then using H\"older inequality
with $p=\frac{2^*}{\beta-2}$, $\frac{1}{p}+\frac{1}{q}=1$, we obtain
\be\label{formula49} \int_{\Gamma_{2}^0} \eta^2 U^{2\delta-1}
F(U)\,dx\leq \left[\int_{\Gamma_{2}^0}
U^{2^*}\,dx\right]^{\frac{1}{p}} \left[\int_{\Gamma_{2}^0}
\eta^{2q}U^{2\delta q}\,dx\right]^{\frac{1}{q}}\leq
V^{\frac{1}{p}}\left[\int_{\Gamma_{2}^0} \eta^{2q}U^{2\delta
q}\,dx\right]^{\frac{1}{q}}.\ee
This last integral can be handled as
follows. Call $\chi=\frac{2^*}{2}$, for simplicity. Because our
hypothesis on $\beta$, we know that $q\in(1,\chi)$. Then, there
exists $\lambda\in(0,1)$ such that $q=\lambda+(1-\lambda)\chi$, and an
interpolation inequality gives: \be\label{interpolation}\left[\int
f^q\right]^{\frac{1}{q}}\leq \left[\int f\right]^{\frac{\lambda}{q}}
\left[\int f^\chi\right]^{\frac{1-\lambda}{q}}=\left[\int
f^\chi\right]^{\frac{1}{\chi}} \lp \left[\int f\right]\left[\int
f^\chi\right]^{-\frac{1}{\chi}}\rp^{\frac{\lambda}{q}}.\ee
Since
$\frac{\lambda}{q}<1$, Young's inequality reads
$$z^{\frac{\lambda}{q}}\leq C_\epsilon z+\epsilon,$$
for $\epsilon$ small. If we substitute $z=\left[\int f
\right]\left[\int f^\chi\right]^{-\frac{1}{\chi}}$ above, together
with \eqref{interpolation}, we arrive at
$$
\left[\int f^q\right]^{\frac{1}{q}}\leq \epsilon \left[\int f^\chi\right]^{\frac{1}{\chi}}+C_\epsilon \int f.
$$
Then from \eqref{formula49} it follows that
\be\label{formulaI2}I_2\leq V^{\frac{1}{p}} \left\{ \epsilon
\lp\int_{\Gamma_2^0} (\eta
U^\delta)^{2^*}\,dx\rp^{\frac{2}{2^*}}+C_\epsilon \int_{\Gamma_2^0}
\eta^2 U^{2\delta}\,dx\right\},\ee
where $\epsilon$ will be chosen later and will depend on the value of $\alpha,\delta$.

We go back now to the main iteration formula \eqref{formula-moser}.
It is clear from \eqref{formula-trace}, that the first integral of
the right hand side of the formula for $I_2$ \eqref{formulaI2} can
be absorbed into the left hand side of \eqref{formula-moser}, and
using \eqref{formula-Sobolev} and \eqref{formula-trace} we get that
$$\lp\int_{\Gamma_1^0} U^{\delta{2^*}}\,dx\rp^{\frac{2}{2^*}}+\lp\int_{B_1^+} U^{2k\delta}\,dxdy\rp^{\frac{1}{k}}\leq
C(\delta)\left[\int_{\Gamma_2^0} U^{2\delta}\,dx+ \int_{B_2^+}
U^{2\delta}\,dxdy\right],$$
for some suitable choice of $\epsilon$.
Or switching notation from $2\delta$ to
$\delta$, \be\label{formula48}\lp\int_{\Gamma_1^0}
U^{\delta\chi}\,dx\rp^{\frac{1}{\chi}}+\lp\int_{B_1^0}
U^{k\delta}\,dxdy\rp^{\frac{1}{k}}\leq
C(\delta)\left[\int_{\Gamma_2^0} U^{\delta}\,dx+ \int_{B_2^0}
U^{\delta}\,dxdy\right].\ee Next, because we will always have $\delta>1$,
we can use that
$$C_1(a^{\frac{1}{\delta}}+b^{\frac{1}{\delta}})\leq (a+b)^{\frac{1}{\delta}}\leq C_2(a^{\frac{1}{\delta}}+b^{\frac{1}{\delta}}),$$
so from \eqref{formula48} we get that
$$\norm{U}_{L^{\chi\delta}(\Gamma_1^0)}+\norm{U}_{L^{k\delta}(B^+_1,y^a)}\leq \norm{U}_{L^{\delta}(\Gamma_2^0)}+\norm{U}_{L^{\delta}(B^+_2,y^a)}.$$
For simplicity, we set
$$\theta:=\min\{\chi,k\}>1,$$
and
$$\Phi(\delta, R):=\lp\frac{1}{R^n}\rp^{\frac{1}{\delta}}\norm{U}_{L^{\delta}(\Gamma_1^0)}+
\lp\frac{1}{R^{n+1+a}}\rp^{\frac{1}{\delta}}\norm{U}_{L^{\delta}(B^+_1,y^a)}.$$
Then, after explicitly writing all the constants involved, formula \eqref{formula48} simply reduces to
$$\Phi(\theta\delta, 1)\leq \left[C(1+\delta)^\sigma\right]^{\frac{2}{\delta}} \Phi(\delta,2),$$
for some positive number $\sigma$.
It is clear that the same proof works if we replace $B_1,B_2$ by $B_{R_1}$, $B_{R_2}$. The only difference is in \eqref{formulaI1}, where we need to estimate $|\nabla \eta|\leq C(R_2-R_1)^{-1}$. Thus we would obtain
\be\label{iteration-step}\Phi(\theta\delta, R_1)\leq
\left[\frac{C(1+\delta)^\sigma}{R_2-R_1}\right]^{\frac{2}{\delta}}
\Phi(\delta,R_2).\ee Now we iterate equation \eqref{iteration-step}:
set $R_m= 1+\frac{1}{2^m}$ and $\theta_m=\theta^m \bar p$. Then
\be\label{formula47}\Phi(\theta_m,1)\leq \Phi(\theta_m,R_m)\leq \lp
c_1 \theta\rp^{c_2 \sum_{i=0}^{m-1} \frac{i}{\theta^i} }\Phi(\bar
p,2)\leq C \Phi(\bar p,2),\ee for some constant $C$ because the
series $\sum_{i=0}^\infty \frac{i}{\theta^i}$ is convergent.

Finally, note that
$$\sup_{\Gamma_1^0} U =\lim_{\delta\to\infty} \norm{U}_{L^\delta(\Gamma_1^0)}, \quad \sup_{B_1^+} U =\lim_{\delta\to\infty} \norm{U}_{L^\delta(B_1^+,y^a)},$$
so that \eqref{formula47} is telling us that
$$ \sup_{B_1^+} U+\sup_{\Gamma_1^0} U\leq C \left[\norm{U}_{L^{\bar p}(B_2,y^a)}+\norm{U}_{L^{\bar p}(\Gamma_2^0)}\right].$$
Rescaling to a ball of radius $R$ concludes the proof of the
theorem.
\end{proof}

The next main ingredient is the proof of the positivity of a
solution to \eqref{degenerate-equation}. We observed that a Hopf
lemma, some version of which was known for the Euclidean half space
case (Proposition 4.10 in \cite{Cabre-Sire:I}), can be obtained for
the uniformly degenerate elliptic equation \eqref{degenerate}. This
nice Hopf's lemma turns out to be one of the keys for us in this
paper. It is interesting to observe a different behavior between the
cases $\gamma\in(0,1/2)$ and $\gamma\in[1/2,1)$ in our proof - this
dichotomy does not seem to appear in the flat case in
\cite{Cabre-Sire:I}.

We continue to use the setting as in Proposition
\ref{new-defining-function}. Let $p_0\in\partial X$ and $(x, y)$ be
the local coordinate at $p_0$ for $\bar X$ with $x(p_0) = 0$, where
$x$ is the normal coordinate at $p_0$ with respect to the metric
$\hat h$ on the boundary $M^n$.

\begin{thm}\label{Hopf} Suppose that $U$ is a nonnegative solution to \eqref{degenerate}
in $X^{n+1}$. Then, for sufficiently small $r_0$, if $U(q_0) = 0$
for $q_0 \in \Gamma_{r_0}^0 \setminus
\overline{\Gamma_{\frac{1}{2}r_0}^0}$ and $U > 0$ on $\partial
\Gamma_{\frac 1 2 r_0}^0$ on the boundary $M^n$, then
\begin{equation}\label{normal-derivative}
y^a \partial_y U|_{q_0} > 0.
\end{equation}
\end{thm}

\begin{proof}
First we assume that $\gamma \in [1/2,1)$, i.e., $a\in(-1,0]$. We
consider a positive function
\begin{equation}\label{test-function}
W = y^{-a}(y +Ay^2)(e^{-B|x|} - e^{-Br_0}).
\end{equation}
To calculate $\text{div}(y^a\nabla W)$ in the metric $\bar g^*$ we
first calculate from Proposition \ref{new-defining-function} that
$$
\bar g^*  = \lp 1 + \alpha_1 y\rp dy^2 + \lp 1 + \alpha_2 y \rp\hat
h + o(y)
$$
for some constants $\alpha_1, \alpha_2$ and
$$
\det \bar g^* = \det \hat h  \lp 1 + \alpha_3 y \rp + o(y),
$$
for some constant $\alpha_3$. Then
$$
\text{div} (y^a\nabla W) = I_1 + I_2 +I_3 +I_4,
$$
where
$$
\aligned I_1  & = \frac 1{\sqrt{\det \bar g^*}}\partial_y \lp
\sqrt{\det \bar g^*}(\bar g^*)^{yy}((1-a)+(2-a)y A)(e^{-B|x|} -
e^{-Br_0})\rp \\ & = \lp \alpha_4 + (2-a)A + o(1)\rp(e^{-B|x|} -
e^{-Br_0}),\endaligned
$$
$$
\aligned I_2 &= \frac 1{\sqrt{\det \bar
g^*}}\partial_{x^k}\lp\sqrt{\det \bar g^*}(\bar
g^*)^{ky}((1-a)+(2-a)y A)(e^{-B|x|} - e^{-Br_0})\rp\\ & =
o(1)(e^{-B|x|} - e^{-Br_0}) + o(y) B e^{-Br},\endaligned
$$
for some constant $\alpha_4$,
$$
I_3 = \frac 1{\sqrt{\det \bar g^*}}\partial_{y}\lp\sqrt{\det \bar
g^*}(\bar g^*)^{yk}(y +y^{2}A)\partial_{x^k}(e^{-B|x|} -
e^{-Br_0})\rp = o(y)Be^{-Br},
$$
and
$$
\aligned I_4  & = \frac {y +y^{2}A}{\sqrt{\det \bar
g^*}}\partial_{x^k}\lp\sqrt{\det \bar g^*}(\bar
g^*)^{kj}\partial_{x^j}(e^{-B|x|} - e^{-Br_0})\rp \\ & = \frac {y
+y^{2}A}{\sqrt{\det
\bar g^*}}\partial_{x^k}\lp\sqrt{\det \bar g^*}(\bar g^*)^{kj}\lp-\frac {x_j}rBe^{-Br}\rp \rp\\
& =y B^2 e^{-Br} + o(y) B^2e^{-Br} + yB^2o(r^2) e^{-Br} + o(y)
Be^{-Br}.
\endaligned.
$$
Thus
$$
\aligned \text{div} (y^a\nabla W) & = \lp\alpha_4 + (2-a)A  +
o(1)\rp(e^{-B|x|} - e^{-Br_0}) \\ & \quad\quad + (B^2 +
o(1)B)ye^{-Br}.\endaligned
$$
We remark here that all constants $\alpha$'s can be explicit, but it
would not be any more use. Take $r_0$ sufficiently small and $A$
and $B$ sufficiently large so that
$$
\text{div} (y^a\nabla W) \geq 0
$$
provided that $a\leq 0$. Now we know
$$
\text{div} \lp y^a\nabla (U - \epsilon W)\rp\leq 0
$$
in $\lp\Gamma_{r_0}^0\setminus \overline{ \Gamma_{\frac 1 2
r_0}^0}\rp\times (0, r_0)$ for all $\epsilon>0$, and moreover
$$
U - \epsilon W \geq 0
$$
on $\partial \left\{\lp \Gamma_{r_0}^0 \setminus \overline
{\Gamma_{\frac 12r_0}^0}\rp\times (0, r_0)\right\}$, provided we
choose $\epsilon$ appropriately small. Therefore, due to the maximum
principle we know that
$$
U - \epsilon W > 0
$$
in $\lp \Gamma_{r_0}^0 \setminus \overline {\Gamma_{\frac
12r_0}^0}\rp\times (0, r_0)$. Thus, when $U(x(q_0), 0)= 0$, we have
$$
y^a\partial_y (U - \epsilon W)|_{(x(q_0),0)} \geq 0,
$$
which implies
$$
y^a\partial_y U|_{(x(q_0),0)} \geq \epsilon y^a\partial_y W
|_{(x(q_0),0)} =\epsilon (1-a)(e^{-B|x(q_0)|} - e^{-Br_0}) > 0,
$$
as desired.

When $a \in (0, 1)$, or equivalently, $\gamma\in (0, \frac 12)$, we
instead use the function
$$
W = y^{-a}(y + A y^{2-a})(e^{-B|x|} - e^{-Br_0}).
$$
Then a similar calculation will prove that the conclusion still
holds.
\end{proof}

Positivity of solutions for \eqref{degenerate} is now clear:

\begin{cor}\label{Hopf-c} Suppose that $U\in \mathcal C^2(X)\cap \mathcal C(\bar X)$ is a
nonnegative solution to the equation
$$
\left\{\aligned \text{div}(y^a\nabla U) & = 0 \quad\text{in } (X, \
\bar g^*),\\ y^a\partial_y U & = F(U) \quad \text{on }
M,\endaligned\right.
$$
where $F(0) = 0$. Then $U>0$ on $\bar X$ unless $U\equiv 0$.
\end{cor}

\begin{proof} First, $U>0$ in $X$, and $U$ is not identically zero on the boundary
if it is not identically zero on $\bar X$. Then, on the boundary,
the set where $U$ is positive is nonempty and open. Hence, if the
set where $U$ vanishes is not empty, then, for any small number
$r_0$, there always exist points $p_0$ and $q_0$ as given
in the assumptions of Theorem \ref{Hopf}. Thus we would arrive at
the contradiction from Theorem \ref{Hopf}.
\end{proof}


\section{The $\gamma$-Yamabe problem}\label{section-Yamabe}

Now we are ready to set up the fractional Yamabe problem for
$\gamma\in (0,1)$. On the conformal infinity $(M^n, \ [\hat h])$ of
an asymptotically hyperbolic manifold $(X^{n+1}, \ g^+)$, we
consider a scale-free functional on metrics in the class $[\hat h]$
given by \be\label{functional0} I_\gamma[\hat h] = \frac {\int_M
Q^{\hat h}_\gamma \,dv_{\hat h}}{(\int_M \,dv_{\hat h})^\frac
{n-2\gamma}n}. \ee Or, if we set a base metric $\hat h$ and write a
conformal metric
$$
\hat h_w=w^{\frac{4}{n-2\gamma}}\hat h,
$$
then \be\label{functional1} I_\gamma[w,\hat h]=\frac{\int_M w
P_\gamma^{\hat h} (w) \,dv_{\hat h}}{ \lp \int_M w^{2^*}\,dv_{\hat
h}\rp^{\frac{2}{2^*}}}\ee where $2^*=\frac{2n}{n-2\gamma}$. We will
call $I_\gamma$ the $\gamma$-Yamabe functional.

The $\gamma$\emph{-Yamabe problem} is to find a metric in the
conformal class  $[\hat h]$ that minimizes the $\gamma$-Yamabe
functional $I_\gamma$. It is clear that a metric $\hat h_w$, where
$w$ is a minimizer of $I_\gamma[w,\hat h]$, has a constant
fractional scalar curvature $Q_\gamma^{\hat h_w}$, that is,
\be\label{yamabe} P_\gamma^{\hat h}(w)=c
w^{\frac{n+2\gamma}{n-2\gamma}}, \quad w>0, \ee for some constant
$c$ on $M$.

This suggests that we define the $\gamma$-Yamabe constant
\be\label{y-constant}\Lambda_\gamma (M,[\hat h])=\inf \left\{
I_\gamma[h] : h\in [\hat h]\right\}.\ee
It is then apparent that $\Lambda_\gamma (M,[\hat h])$ is an invariant on the conformal class $[\hat h]$ when $g^+$ is fixed. \\

In the mean time, based on Proposition \ref{prop-Chang-Gonzalez}, we
set \be\label{nicer-functional-1} I^*_\gamma[U,\bar
g]=\frac{d^*_\gamma \int_X \rho^a\abs{\nabla U}_{\bar g}^2
\,dv_{\bar g} + \int_X E(\rho)|U|^2\,dv_{\bar g}}{\lp \int_M
|U|^{2^*} \,dv_{\hat h}\rp^{\frac{2}{2^*}}},\ee or similarly, using
Proposition \ref{new-defining-function}, we may set
\be\label{nicer-functional} I^*_\gamma[U,\bar g^*]=\frac{d^*_\gamma
\int_X y^a\abs{\nabla U}_{\bar g^*}^2 \,dv_{\bar g^*} + \int_M
Q^{\hat h}_\gamma |U|^2\,dv_{\hat h}}{\lp \int_M |U|^{2^*}
\,dv_{\hat h}\rp^{\frac{2}{2^*}}}.\ee It is obvious that it is
equivalent to solve the minimizing problems for $I_\gamma$ and
$I_\gamma^*$. But a very pleasant surprising is that this
immediately tells us that \be\label{good-gamma-Yamabe}
\Lambda_\gamma (X, \ [\hat h])=\inf \left\{ I^*_\gamma[U,\bar g] :
U\in W^{1,2}(X,y^a) \right\} \ee (please see the definitions and
discussions of the weighted Sobolev spaces in Section \ref{trace}).
Note that one has that $I_\gamma^*[|U|] \leq I_\gamma^*[U]$, to
handle positivity issues. Therefore we have

\begin{lemma}
Suppose that $U$ is a minimizer of the functional
$I^*_\gamma[\cdot,\bar g]$ in the weighted Sobolev space $W^{1,2}(X,
y^a)$ with $\int_M |TU|^{2^*} \,dv_{\hat h} = 1$. Then its trace $w
= TU \in H^\gamma(M)$ solves the equation
$$
P_\gamma^{\hat h}(w)= \Lambda_\gamma (X, \ [\hat
h])w^{\frac{n+2\gamma}{n-2\gamma}}.
$$
\end{lemma}

To resolve the $\gamma$-Yamabe problem is to verify $I_\gamma$ has a
minimizer $w$, which is positive and smooth. But before launching
our resolution to the $\gamma$-Yamabe problem we  are first due to
discuss the sign of the $\gamma$-Yamabe constant. These statements
are familiar and easy ones for the Yamabe problem but not so easy at
all for the $\gamma$-Yamabe problem, where the conformal fractional
Laplacians are just pseudo-differential operators. One knows that
eigenvalues and eigenfunctions of the conformal fractional
Laplacians are even more difficult to study than the
differential operators. There are some affirmative results analogous
to the conformal Laplacian proven in \cite{Guillarmou-Qing} when the
Yamabe constant of the conformal infinity is assumed to be positive.
Here we will take the advantage of our Hopf's Lemma and the
interpretation of the conformal fractional Laplacians through extensions
provided in Proposition \ref{new-defining-function}.

For each  $\gamma\in(0,1)$ we know that each conformal fractional
Laplacian is self-adjoint (cf.
\cite{Graham-Zworski:scattering-matrix},
\cite{Fefferman-Graham:Q-curvature}). Hence we may look for the first
eigenvalue $\lambda_1$ by minimizing the quotient
\be\label{b-quotient} \frac {\int_M wP^{\hat h}_\gamma w \,dv_{\hat
h}}{\int_M w^2 \,dv_{\hat h}}. \ee

Moreover, again in the light of Proposition \ref{new-defining-function}, it is equivalent to minimizing \be\label{e-quotient}
\frac{d^*_\gamma \int_X y^a\abs{\nabla U}_{\bar g^*}^2 \,dv_{\bar
g^*} + \int_M Q^{\hat h}_\gamma |U|^2\,dv_{\hat h}}{\int_M |U|^2
\,dv_{\hat h}}.\ee We arrive at the eigenvalue equation:
$$
P_\gamma^{\hat h} w=\lambda_1 w,\quad \mbox{on } M.
$$
Or, equivalently,
\be\label{eigenvalue-extension}\left\{\begin{split}
\divergence \lp y^a \nabla U\rp &=0 \quad\mbox{in } (X, \ \bar g^*),\\
-d_\gamma^*\lim_{y\to 0} y^a \partial_y U+Q_\gamma^{\hat h}
U&=\lambda_1 U \quad\mbox{on } M,
\end{split}\right.\ee
As a consequence of Proposition \ref{new-defining-function} and
Theorem \ref{Hopf} we have:

\begin{thm}\label{eigen} Suppose that $(X^{n+1}, \ g^+)$ is an asymptotically hyperbolic manifold.
For each $\gamma\in (0, 1)$ there is a smooth, positive first
eigenfunction for $P_\gamma^{\hat h}$ and the first eigenspace is of
dimension one, provided $H=0$ when $\gamma\in (\frac 12, 1)$.
\end{thm}

\begin{proof}
We use the variational characterization \eqref{e-quotient} of the
first eigenvalue. We first observe that one may always assume there
is a nonnegative minimizer for \eqref{e-quotient}. Then regularity
and the maximum principle in Section 3 insure that such a first eigenfunction
is smooth and positive. To show that the first eigenspace is of
dimension 1, we suppose that $\phi$ and $\psi$ are positive first
eigenfunctions for $P_\gamma^{\hat h}$. Then
$$
\aligned P_\gamma^{\hat h_\phi} \frac {\psi}{\phi} &  =
\phi^{-\frac{n+2\gamma}{n-2\gamma}}P_\gamma^{\hat
h}\psi  = \lambda_1 \phi^{-\frac{n+2\gamma}{n-2\gamma}} \psi \\
& = (\phi^{-\frac{n+2\gamma}{n-2\gamma}} P_\gamma^{\hat h}\phi
)\frac {\psi}{\phi}  \\ & = Q_\gamma^{\hat h_\phi} \frac \psi\phi,
\endaligned
$$
where $\hat h_\phi = \phi^{\frac 4{n-2\gamma}}\hat h$. That is, there is a function $U$ satisfying
$$
\left\{\aligned \text{div}(y_\phi^a \nabla U) & = 0 \quad \text{in }
(X ,\ \bar g_\phi^*), \\ \lim_{y_\phi \to 0}y^a_\phi \frac{\partial
U}{\partial y_\phi}U & = 0 \quad \text{on } M,\endaligned\right.
$$
and $U = \frac \psi\phi$ on $M$, where $y_\phi$ and $\bar g^*_\phi$
are associated with $\hat h_\phi$ as $y$ and $\bar g^*$ are
associated with $\hat h$ in Proposition \ref{new-defining-function}
respectively. Replace $U$ by $U - U_m$ for $U_m = \min_{\bar X} U$
and apply Theorem \ref{Hopf} and Corollary \ref{Hopf-c} to conclude
that $U$ has to be a constant. 
\end{proof}

Consequently, we get the following.

\begin{cor} Suppose that $(X^{n+1}, \ g^+)$ is an asymptotically hyperbolic manifold.
Assume that $\gamma \in (0, 1)$ and that $H=0$ when $\gamma \in
(\frac 12, 1)$. Then there are three mutually exclusive
possibilities for the conformal infinity $(M^n, \ [\hat h])$:
\begin{enumerate}
\item The first eigenvalue of $P_\gamma^{\hat h}$ is positive, the
$\gamma$-Yamabe constant is positive, and $M$ admits a metric in
$[\hat h]$ that has pointwise positive fractional scalar curvature.

\item The first eigenvalue of $P_\gamma^{\hat h}$ is negative, the
$\gamma$-Yamabe constant is negative, and $M$ admits a metric in
$[\hat h]$ that has pointwise negative fractional scalar curvature.

\item The first eigenvalue of $P_\gamma^{\hat h}$ is zero, the
$\gamma$-Yamabe constant is zero, and $M$ admits a metric in $[\hat
h]$ that has vanishing fractional scalar curvature.
\end{enumerate}
\end{cor}

\begin{proof} First of all it is obvious that the sign of the first
eigenvalue of the conformal fractional Laplacian $P_\gamma^{\hat h}$
does not change within the conformal class due to the conformal
covariance property of the conformal fractional Laplacian. The three
possibilities are distinguished by the sign of the first eigenvalue
$\lambda_1$ of the conformal fractional Laplacian $P_\gamma^{\hat
h}$. Because, if $\phi$ is the positive first eigenfunction of
$P_\gamma^{\hat h}$, then
$$
Q_\gamma^{\hat h_\phi} = \lambda_1^{\hat h}\phi^{-\frac
{4\gamma}{n-2\gamma}}
$$
where $\hat h_\phi = \phi^{\frac 4{n-2\gamma}}\hat h$.
\end{proof}


\section{Weighted Sobolev trace inequalities}\label{trace}

Let us continue in the setting provided by Proposition
\ref{new-defining-function}. On the compact manifold $M^n$, for
$\gamma\in(0,1)$, we recall the fractional order Sobolev space
$H^\gamma(M)$, with its usual norm
$$\norm{w}^2_{H^\gamma(M)}:=\norm{w}^2_{L^2(M)}+ \int_M
w(-\Delta_{\hat h})^\gamma w\,dv_{\hat h}.$$ An equivalent norm on
this space is
$$
\norm{w}^2_{H^\gamma(M)}: =A\norm{w}^2_{L^2(M)}+\int_M w
P_\gamma^{\hat h} w\,dv_{\hat h},
$$
for some appropriately large number $A$, since $P_\gamma^{\hat h}$
is an elliptic pseudo-differential operator of order $2\gamma$ with
its principal symbol being the same as that of $(-\Delta_{\hat
h})^\gamma$.

Note that in $\mathbb R^n$, this Sobolev norm can be easily written
in terms of the Fourier transform as \be\label{Sobolev-norm}
\norm{w}^2_{H^\gamma(\mathbb R^n)}=\int_{\mathbb R^n} (
1+|\xi|^2)^\gamma \hat w^2(\xi)\,d\xi. \ee

We  would  also like to recall the definition of the weighted Sobolev
spaces. For $\gamma\in(0,1)$ and $a = 1 - 2\gamma$, consider the norm
$$
\norm{U}_{W^{1,2}(X,y^a)}^2= \int_X y^a|\nabla U|_{\bar g^*}^2
\,dv_{\bar g^*}+\int_{X} y^a U^2  \,dv_{\bar g^*}.
$$
The following is then known.
\begin{lemma}\label{Nek} There exists a unique linear bounded operator
$$
T: W^{1,2}(X,y^a)\to H^\gamma (M)
$$
such that $TU=U|_M$ for all $U\in\mathcal C^\infty(\bar X)$, which
is called the trace operator.
\end{lemma}
Lemma \ref{Nek} was explored by Nekvinda
\cite{Nekvinda:characterization-tracesW} in the case when $X$ is a subset of $\mathbb R^{n+1}$ and $M^n$ a piece of its
boundary;  see also
\cite{Mazja}. It then takes some standard argument to derive the
Lemma \ref{Nek} from, for instance,
\cite{Nekvinda:characterization-tracesW}.

The classical Sobolev trace inequality on Euclidean space is well
known (see, for instance, Escobar \cite{Escobar:sharp-constant}),
and  reads:
\begin{equation}\label{sobolev-trace-0}
\lp \int_{\mathbb R^n}
\abs{Tu}^{\frac{2n}{n-1}}\,dx\rp^{\frac{n-1}{2n}}\leq C(n)
\lp\int_{\mathbb R^{n+1}_+} \abs{\nabla u}^2 \,
dxdy\rp^{\frac{1}{2}}
\end{equation}
where the constant $C(n)$ is sharp and the equality case is
completely characterized. This corresponds to $a = 0$ for our cases.
The same result is true for any other real $a \in (-1,1)$. Indeed
there are general Weighted Sobolev trace inequalities. Let us first
recall the well known fractional Sobolev inequalities. They were
considered first in the remarkable paper by Lieb
\cite{Lieb:sharp-constants} (see also the more recent \cite{Frank-Lieb:HLS}, \cite{Cotsiolis-Tavoularis:best-constants}, or the survey \cite{DiNezza-Palatucci-Valdinoci}):

\begin{lemma} \label{thm-embedding} Let $0<\gamma<n/2$, $2^*=\frac{2n}{n-2\gamma}$. Then, for all
$w\in H^\gamma(\R^n)$ we have
\be\label{Sobolev}\norm{w}_{L^{2^*}(\R^n)}^2\leq
S(n,\gamma)\|(-\Delta)^{\frac{\gamma}{2}}w\|^2_{H^\gamma(\R^n)}=S(n,\gamma)\int_{\R^n}
w (-\Delta)^\gamma w\,dx,\ee where
$$S(n,\gamma)=2^{-2\gamma}\pi^{-\gamma}\frac{\Gamma\lp\frac{n-2\gamma}{2}\rp}
{\Gamma\lp\frac{n+2\gamma}{2}\rp}
\left[\frac{\Gamma(n)}{\Gamma\lp\frac{n}{2}\rp}\right]^{\frac{2\gamma}{n}}=
\frac{\Gamma\lp\frac{n-2\gamma}{2}\rp}{\Gamma\lp\frac{n+2\gamma}{2}\rp}
|vol(S^n)|^{-\frac{2\gamma}{n}}.
$$
We have equality in \eqref{Sobolev} if and only if
$$w(x)=c\lp\frac{\mu}{\abs{x-x_0}^2+\mu^2}\rp^{\frac{n-2\gamma}{2}},\quad x\in\R^n,$$
for $c\in\mathbb R$, $\mu>0$ and $x_0\in\R^n$ fixed.
\end{lemma}

Note that we may interpret the above inequality as a calculation of
the best $\gamma$-Yamabe constant on the standard sphere as the
conformal infinity of the Hyperbolic space. Namely, if $g_c$ is the
standard round metric on the unit sphere, \be\label{Sobolev-sphere}
\norm{w}^2_{L^{2^*}(S^n)}\leq S(n,\gamma)\int_{S^n} w
P_{\gamma}^{g_c} w \,dv_{g_c}.\ee Such an inequality for the sphere
case was also considered independently by Beckner
\cite{Beckner:sharp-Sobolev}, Branson
\cite{Branson:sharp-inequalities}, and  Morpurgo \cite{Morpurgo},
in the setting of interwining operators. Indeed, we have the
following explicit expression for $P_\gamma^{S^n}$: \bee
P^{S^n}_\gamma=\frac{\Gamma\lp B+\gamma+\tfrac{1}{2}\rp}{\Gamma\lp
B-\gamma+\tfrac{1}{2}\rp},\quad \mbox{where}\quad
B:=\sqrt{-\Delta_{S^n}+\lp\tfrac{n-1}{2}\rp^2}.\eee It is clear from
\eqref{Sobolev-sphere} that
\be\label{Sobolev-constant}\Lambda_\gamma(S^n,[g_c])=\frac{1}{S(n,\gamma)}.\ee

Sobolev trace inequalities can be obtained by the composition of the
trace theorem and the Sobolev embedding theorem above. There
have been some related works that deal with these types of energy
inequalities, for instance, Nekvinda
\cite{Nekvinda:characterization-tracesW}, Gonz\'alez
\cite{phase-transitions}, and Cabr\'e-Cinti
\cite{Cabre-Cinti:energy-estimates}. In particular, in the light of
the work of Caffarelli and Silvestre \cite{Caffarelli-Silvestre} and
Lemma \ref{thm-embedding}, we easily see the more general form of
\eqref{sobolev-trace-0} as follows:

\begin{cor}\label{cor-trace-inequality}
Let  $w\in H^\gamma(\R^n)$, $\gamma \in (0, 1)$, $a = 1 - 2\gamma$,
and $U\in W^{1,2}(\RR,y^a)$ with trace $TU=w$. Then
\be\label{embedding}\norm{w}_{L^{2^*}(\R^n)}^2\leq \bar
S(n,\gamma)\int_{\RR} y^a \grad U 2\,dxdy,\ee where
\be\label{S-bar}\bar S(n,\gamma):= d^*_\gamma S(n,\gamma).\ee
Equality holds if and only if
$$w(x)=c\lp\frac{\mu}{\abs{x-x_0}^2+\mu^2}\rp^{\frac{n-2\gamma}{2}},\quad x\in\R^n,$$
for $c\in\mathbb R$, $\mu>0$ and $x_0\in\R^n$ fixed, and $U$
is its Poisson extension of $w$ as given in \eqref{Poisson}.
\end{cor}

In the following lines we take a closer look at the extremal functions that attain the best constant in the inequality above. On $\mathbb R^n$ we fix
\be\label{diffeo-sphere}
w_\mu(x):=\lp\frac{\mu}{\abs{x}^2+\mu^2}\rp^{\frac{n-2\gamma}{2}},\ee
these correspond to the conformal diffeomorphisms of the sphere. We
set \be\label{diffeo-sphere-extension}U_\mu=K_\gamma *_x w_\mu \ee
as given in \eqref{Poisson}.
Then we have the equality
$$\norm{w_\mu}_{L^{2^*}(\mathbb R^n)}^2= \bar S(n,\gamma)\int_{\RR}
y^a \grad {U_\mu} 2 \,dxdy.$$
It is clear that
\be\label{scaling-properties}w_\mu(x)=\frac{1}{\mu^{\frac{n-2\gamma}{2}}}w_1
\lp\frac{x}{\mu}\rp,\quad \mbox{and}\quad
U_\mu(x,y)=\frac{1}{\mu^{\frac{n-2\gamma}{2}}}U_1
\lp\frac{x}{\mu},\frac{y}{\mu}\rp.\ee Moreover, $U_\mu$ is the
(unique) solution of the problem
\be\label{extension-Rn}\left\{\begin{split}
\divergence(y^a \nabla U_\mu)=0 &\quad\mbox{in } \mathbb R^{n+1}_+,\\
-\lim_{y\to 0}y^a\partial_y
U_\mu=c_{n,\gamma}(w_\mu)^{\frac{n+2\gamma}{n-2\gamma}}&\quad\mbox{on
} \mR^n,
\end{split}\right.\ee
On the other hand, if we multiply equation \eqref{extension-Rn} by
$U_\mu$ and integrate by parts,
\be\label{formula200}\int_{\mR^{n+1}_+} y^a\grad {U_\mu} 2\,dxdy=
c_{n,\gamma} \int_{\mathbb R^n} (w_\mu)^{2^*}\,dx.\ee Now we compare
\eqref{formula200} with \eqref{embedding}. Using
\eqref{Sobolev-constant} we arrive at
\be\label{formula202}\Lambda(S^n,[g_c])=c_{n,\gamma}d_\gamma^*\left[\int_{\mR^n}
(w_\mu)^{2^*}dx\right]^{\frac{2\gamma}{n}}.\ee

Before the end of this section we calculate the general upper bound
of the $\gamma$-Yamabe constants. Indeed there is a complete
analogue to the case of the usual Yamabe problem (cf.
\cite{Aubin:book1}, \cite{Lee-Parker}). Namely, the following.

\begin{prop}\label{prop-compare-sphere} Let $\gamma\in(0,1)$. Then
$$\Lambda_\gamma(M,[\hat h])\leq \Lambda_\gamma(S^n,[g_c]).$$
\end{prop}

\begin{proof} First of all we will instead use the functional
\eqref{nicer-functional} to estimate the $\gamma$-Yamabe constant
for a good reason. The approach is rather the standard method of
gluing a ``bubble" \eqref{diffeo-sphere} to the manifold $M$ (see, for instance, \cite{Lee-Parker}, Lemma 3.4).

For any fixed $\epsilon>0$, let $B_\epsilon$ be the ball of radius
$\epsilon$ centered at the origin in $\mathbb R^{n+1}$ and
$B_\epsilon^+$ be the half ball of radius $\epsilon$ in $\RR$.
Choose a smooth radial cutoff function $\eta$, $0\leq\eta\leq 1$
supported on $B_{2\epsilon}$, and satisfying $\eta\equiv 1$ on
$B_\epsilon$. Then, consider the function $V=\eta U_\mu$ with its
trace $v=\eta w_\mu$ on $\mathbb R^n$. We have that
\begin{equation}\label{gluing}
\int_{\RR} y^a\grad {V} 2 \,dxdy  \leq (1+\epsilon)\int_{\RR}y^a
\grad {U_\mu} 2\,dxdy +C(\epsilon)\int_{B_{2\epsilon}^+ \backslash
B_\epsilon^+ } U_\mu^2 \,dxdy.
\end{equation}
Note that $w_\mu = O (\mu^{\frac{n-2\gamma}{2}}
\abs{x}^{2\gamma-n})$ in the annulus $\epsilon\leq \abs{x}\leq
2\epsilon$ and $U_\mu$ is $O(\mu^{\frac{n-2\gamma}{2}})$ in the
annulus $B_{2\epsilon}^+ \backslash B_\epsilon^+$. This allows to
estimate the second term in right hand side of \eqref{gluing} by
$O\lp\mu^{n-2\gamma}\rp$ as $\mu\to 0$, for $\epsilon$ fixed. For
the first term in the right hand side of \eqref{gluing} we first use
the fact that $w_\mu$ attains the best constant in the Sobolev inequality, so
\begin{equation}\label{transplant1}
\bar S(n,\gamma)\int_{\RR} y^a \grad {U_\mu} 2\,dx dy  = \lp\int_{\mathbb R^n} w_\mu^{2^*}dx\rp^{\frac{2}{2^*}}\leq \lp
\int_{\mathbb R^n} v^{2^*}dx\rp^{\frac{2}{2^*}}+O(\mu^n).
\end{equation}

Now we need to transplant the function $V$ to the manifold $(\bar
X,\bar g^*)$. Fix a point on the boundary $M$ and use normal
coordinates $\{x_1,\ldots,x_n,y\}$ around it, in a half ball
$B_{2\epsilon}^+$ where $V$ is supported. Two things must be
modified: when $\epsilon\to 0$,
$$\abs{\nabla V}_{\bar g^*}^2=|\nabla V|^2(1+O(\epsilon)),$$
and
$$dv_{\bar g^*}=(1+O(\epsilon))dxdy,$$
so that \bee\begin{split} I_{\epsilon,\mu}&:
=d_\gamma^*\int_{B_{2\epsilon}^+}  y^a |\nabla V|_{\bar g^*}^2
\,dv_{\bar g^*}
+ \int_{\abs{x}\leq 2\epsilon} Q_\gamma^{\hat h}v^2\,dv_{\hat h}\\
& \leq  (1+O(\epsilon)) \lp \int_{B_{2\epsilon}^+} y^a |\nabla V|^2
\,dxdy+ C\int_{\abs{x}<2\epsilon} v^2 \,dx\rp.
\end{split}\eee
It is easily seen that
$$
\int_{\abs{x}<2\epsilon} w_\mu^2\,dx = o(1).
$$
This is a small computation that can be found in Lemma 3.5 of
\cite{Lee-Parker}. Then, from \eqref{transplant1}, fixing $\epsilon$
small and then $\mu$ small, we can get that
$$
I_{\epsilon,\mu}\leq (1+C\epsilon)\lp \frac{1}{S(n,\gamma)}
\norm{v}_{L^{2^*}(M)}^2+C\mu\rp
$$
which implies
$$
\Lambda_\gamma(M,[\hat h]) \leq \frac{1}{S(n,\gamma)} =
\Lambda_\gamma(S^n, [g_c]).$$

\end{proof}

We end this section by remarking that, although most of the results
mentioned here were already known in different contexts, it is
certainly very interesting to put all the analysis and geometry
together in the context of conformal fractional Laplacians and the
associated $\gamma$-Yamabe problems in a way that is analogous to what
has been done on the subject of the Yamabe problem, which becomes
fundamental to the development of geometric analysis.


\section{Subcritical approximations}

In this section we take a well known subcritical approximation
method to solve the $\gamma$-Yamabe problem and prove Theorem
\ref{thm1}. There does not seem to be  any more difficulty than usual after our discussions in previous sections. But, for the
convenience of the readers, we present a brief sketch of the proof.
Similar to the case of the usual Yamabe problem we 
consider the following subcritical approximations to the functionals
$I_\gamma$ and $I_\gamma^*$ respectively. Set
$$
I_\beta[w]=\frac{\int_M wP_\gamma^{\hat h}w\,dv_{\hat h}}{\lp \int_M
w^{\beta} \,dv_{\hat h}\rp^{\frac{2}{\beta}}}
$$
and
$$
I^*_\beta[U]=\frac{d^*_\gamma\int_X y^a\abs{\nabla U}_{\bar g^*}^2
\,dv_{\bar g} +\int_M Q^{\hat h}_\gamma U^2\,dv_{\hat h}}{\lp \int_M
U^{\beta} \,dv_{\hat h}\rp^{\frac{2}{\beta}}}.
$$
for $\beta \in [2, 2^*)$, where $2^* = \frac {2n}{n-2\gamma}$ and
$\gamma\in (0, 1)$. These are subcritical problems and can be solved
through standard variational methods. For clarity we state the
following:
\begin{prop}\label{prop-subcritical} For each $2\leq \beta<2^*$,
there exists a smooth positive minimizer $U_\beta$ for
$I^*_\beta[U]$ in $W^{1,2}(X,y^a)$, which satisfies the equations
\bee\left\{\begin{split}
\divergence \lp y^a \nabla U_\beta\rp & = 0 \quad\mbox{in } (X, \bar g^*),\\
-d_\gamma^*\lim_{y\to 0} y^a U_\beta + Q_\gamma^{\hat h} U_\beta & =
c_\beta U_\beta^{\beta-1}  \quad\mbox{on } M ,
\end{split}\right. \eee where the derivatives are taken with respect to the
metric $\bar g^*$ in $X$ and $c_\beta =$ $ I_\beta^*[U_\beta] = \min
I_\beta^*$. And the boundary value $w_\beta$ of $U_\beta$, which is
a positive smooth minimizer for $I_\beta[w]$ in $H^\gamma(M)$,
satisfies
$$
P_\gamma^{\hat h} w_\beta = c_\beta w_\beta^{\beta - 1}.
$$
\end{prop}
Using a similar argument as in the proof of Lemma 4.3 in
\cite{Lee-Parker} (see also \cite{Aubin:book1}) we have the following.

\begin{lemma}\label{lemma-constants-converge}
If $\text{vol}(M, \hat h) = 1$, then $|c_\beta|$ is non-increasing
as a function of $\beta\in [2, 2^*]$; and if $\Lambda_\gamma(M,
[\hat h]) \geq 0$, then $c_\beta$ is continuous from the left at $\beta
= 2^*$.
\end{lemma}

We now start the proof of Theorem \ref{thm1}. Readers are referred
to \cite{Escobar:conformal-deformation}, \cite{Lee-Parker},
\cite{Schoen-Yau:book} for more details. Instead of applying the
standard Sobolev embedding in the Yamabe problem we apply the
weighted trace ones discussed in the previous section. To ensure
that $U_\beta$ as $\beta\to 2^*$ produces a minimizer for the
$\gamma$-Yamabe problem, we want to establish the a priori estimates
for $U_\beta$. In the light of the discussions in Section
\ref{section-elliptic}, we only need to have a uniform $L^\infty$
bound for $w_\beta$. We will establish the $L^\infty$ bound for
$w_\beta$ by the so-called blow-up method.

Otherwise, assume  there exist sequences $\beta_k\to 2^*$,
$w_k:=w_{\beta_k}$ and $U_k := U_{\beta_k}$, $x_k\in M$ such that
$w_k(x_k)=\max_M\{w_k\}=m_k\to \infty$ and $x_k\to x_0\in M$ as
$k\to \infty$. Take a normal coordinate system centered at $x_0$,
and rescale
$$
V_k(x,y)=m_k^{-1} U_k(\delta_k x+x_k, \delta_k y),
$$
with the boundary value
$$
v_k(x)=m_k^{-1}w_k(\delta_k x+x_k),
$$
where $\delta_k = m_k^{\frac{1-\beta_k}{2\gamma}}$. Then $V_k$ is
defined in a half ball of radius $R_k=\frac{1-\abs{x_k}}{\delta_k}$
and is a solution of \be\label{problem-rescaled}\left\{\begin{split}
\divergence \lp \rho^a \nabla V_k\rp &=0 \quad\mbox{in }B^+_{R_k},\\
-{d_\gamma^*}\lim_{y\to 0}y^a\partial_{y}V_k + (Q^{\hat h}_\gamma)_k
v_k &=c_k v_k^{\beta-1} \quad\mbox{on } B_{R_k},
\end{split}\right.\ee
with respect to the metric $\bar g^*(\delta_k x+x_k, \delta_k y)$,
where
$$
(Q^{\hat h}_\gamma)_k = \delta_k^{1-a} Q^{\hat h}_\gamma(\delta_k x
+ x_k) \to 0.
$$
Due to, for example, $\mathcal C^{2, \alpha}$ a priori estimates for
the rescaled solutions $V_k$, to extract a subsequence if necessary,
we have $V_k\to V_0$ in $ C^{2,\alpha}_{\mbox{loc}}$. Moreover the
metrics $\bar g^*(\delta_k x+x_k, \delta_k y)$ converge to the Euclidean
metric. Hence $V_0$ is a non-trivial, non-negative solution of
\be\label{limit-equation}\left\{\begin{split}
-\divergence \lp y^a \nabla V_0\rp &=0 \quad\mbox{in }\mathbb R^{n+1}_+,\\
-{d_\gamma^*}\lim_{y\to 0}y^a\partial_y V_0&=c_0
V_0^{\frac{n+2\gamma}{n-2\gamma}}  \quad\mbox{on } \mathbb R^n,
\end{split}.\right.\ee
Let $v_0=TV_0$. It is easily seen that
\begin{equation}\label{less-than-one}
\int_{\mathbb R^n} v_0^{2^*}(x)\,dx \leq 1.
\end{equation}
Theorem \ref{Hopf} and Corollary \ref{Hopf-c} then assure that $V_0
> 0$ on $\overline{ \mathbb R^{n+1}_+}$. Therefore we can obtain
\be\label{eq1}\int_{\RR}y^a \grad {V_0} 2\,dxdy=c_0
d^*_\gamma\int_{\mathbb R^n} v_0^{2^*}(x)\,dx.\ee It is then obvious
that $c_0 > 0$, that is, $c_0 = \Lambda_\gamma(M,[\hat h])$ in the
light of Lemma \ref{lemma-constants-converge}. Moreover, by the
trace inequalities from Lemma \ref{cor-trace-inequality}, we have
\begin{equation}\label{eq2}
\lp\int_{\mathbb R^n}v_0^{2^*}(x)\,dx\rp^{\frac{2}{2^*}} \leq \bar
S({n,\gamma})\int_{\RR}y^a\grad {V_0} 2\,dxdy.
\end{equation}
Then \eqref{less-than-one}, \eqref{eq1} and \eqref{eq2}, together
with the definition of $\Lambda_\gamma(S^n,[g_c])$ in
\eqref{Sobolev-constant} contradict the initial hypothesis
\eqref{condition}.

Once we have a uniform $L^\infty$ estimate, by the regularity
theorems in Section \ref{section-elliptic} we may extract a
subsequence if necessary and pass to a limit $U_0$, whose boundary
value $w_0$ satisfies
\begin{equation}\label{f-equ}
P_\gamma^{\hat h} w_0 =\Lambda w_0^{2^*-1}, \quad I_\gamma[w_0
]=\Lambda, \quad \Lambda=\lim c_\beta.
\end{equation}
Theorem \ref{Hopf} and Corollary \ref{Hopf-c} also ensure that
$w_0>0$ on $M$. It remains to check that $\Lambda = \Lambda_\gamma(M, [\hat h])$. However,
this is a direct consequence of Lemma \ref{lemma-constants-converge}
when $\Lambda_\gamma(M, [\hat h])\geq 0$. Meanwhile it is easily
seen that by the definition of the $\gamma$-Yamabe constants and
\eqref{f-equ} that $\Lambda$ can not be less than $\Lambda_\gamma(M,
[\hat h])$. Hence it is also implied that $\Lambda =
\Lambda_\gamma(M, [\hat h])$ by Lemma \ref{lemma-constants-converge}
when $\Lambda_\gamma(M, [\hat h]) < 0$. Thus, in any case, $w_0$ is
a minimizer of $I_\gamma$, as desired.
\qed

\section{A sufficient condition}

In this section we give the proof of Theorem \ref{thm2}, which
provides a sufficient condition for the resolution of the
$\gamma$-Yamabe problem. Here the precise structure of the
metric will play a crucial role since a careful computation of the asymptotics is required, following the calculation in
\cite{Escobar:conformal-deformation}. The section is divided into
two parts: the first  contains the necessary estimates on the
Euclidean case, while in the second we go back to the geometry
setting and finish the proof of the theorem.

\subsection{Some preliminary results on $\mathbb R^{n+1}_+$}

Here we consider the divergence equation \eqref{div} on $\mathbb
R^{n+1}_+$, as understood in \cite{Caffarelli-Silvestre},
\cite{phase-transitions}. The main point is that by using the Fourier
transform, a solution to this problem can be written in terms on its
trace value on $\mathbb R^n$ and the well known Bessel functions.
Indeed, let $U$ be a solution of
\begin{equation}\label{equation111}
\left\{\begin{split}\divergence(y^a\nabla U)& =0 \quad\mbox{ in }\R^{n+1}_+, \\
U(x,0)& =w \quad \mbox{ on }\R^n\times \{0\},
\end{split}\right.
\end{equation}
or equivalently, $U=K_\gamma*_x w$, where $K_\gamma$ is the Poisson
kernel as given in \eqref{Poisson}.

The main idea is to reduce \eqref{equation111} to an ODE by taking
Fourier transform in $x$. We obtain \bee\left\{\begin{split}
-\abs{\xi}^2\hat u(\xi,y)+\frac{a}{y}\hat u_y(\xi,y)+\hat u_{yy}(\xi,y)&=0, \\
\hat U(\xi,0)&=\hat w(\xi),
\end{split}\right.\eee
that is an ODE for each fixed value of $\xi$.

On the other hand, consider the solution $\varphi:[0,+\infty)\to \R$
of the problem \be\label{equation-phi}-\varphi(y)+\frac{a}{y}
\varphi_y(y)+ \varphi_{yy}(y)=0,\ee subject to the conditions
$\varphi(0)=1$ and $\lim\limits_{t\to +\infty} \varphi (t)= 0$. This
is a Bessel function and its properties are summarized in Lemma
\ref{lemma-Bessel}. Then we have that
\be\label{Fourier-transform}\hat U(\xi,y)=\hat
w(\xi)\varphi(\abs{\xi}y).\ee

For a review of Bessel functions (see, for
instance, Lemma 5.1 in \cite{phase-transitions}, or section 9.6.1.
in \cite{Abramowitz-Stegun}):

\begin{lemma}\label{lemma-Bessel}
Consider the following ODE in the variable $y>0$:
$$- \varphi(y)+\frac{a}{y}\varphi_y(y)+\varphi_{yy}(y)=0,$$
with boundary conditions $\varphi(0)=1$, $\varphi(\infty)=0$. Its
solution can be written in terms of Bessel functions:
\bee\label{solution-ODE}\varphi(y)=c_1y^{\gamma}\mathcal
K_\gamma(y),\eee where $\mathcal K_\gamma$ is the modified Bessel
function of the second kind that has asymptotic behavior
\begin{align*}\mathcal K_\gamma(y)&\sim \frac{\Gamma(\gamma)}{2}\lp\frac{2}{y}\rp^{\gamma},\quad\mbox{when }y\to 0^+,\\
\mathcal K_\gamma(y)&\sim
\sqrt{\frac{\pi}{2y}}\;e^{-y},\quad\mbox{when }y\to +\infty,
\end{align*}
for a constant
$$c_1=\frac{2^{1-\gamma}}{\Gamma(\gamma)}.$$
\end{lemma}

Now we are ready to prove the main technical lemmas in the proof of
Theorem \ref{thm2}.  More precisely, we will explicitly compute
several energy terms through Fourier transforms, thanks to expression
\eqref{Fourier-transform}. Such precise computation is needed in
order to obtain the exact value of the constant \eqref{cst}. For the
rest of the section, we denote $\grad {U}
2=\lp\partial_{x_1}U\rp^2+\ldots+\lp\partial_{x_n}U\rp^2+\lp\partial_y
U\rp^2$, and  $|\nabla_x
{U}|^2=\lp\partial_{x_1}U\rp^2+\ldots+\lp\partial_{x_n}U\rp^2$.

\begin{lemma} \label{lemma-compare-integrals}
Given $w\in H^{\gamma}(\mathbb R^n)$, let $U=K_\gamma * w$ defined
on $\mathbb R^{n+1}_+$. Then
\begin{align}\label{notation-A}
\mathcal A_1(w)&:=\int_{\halfspace}y^{a+2}\grad {U} 2 \,dxdy=d_1\int_{\R^n}\abs{\hat w(\xi)}^2 \abs{\xi}^{2(\gamma-1)}\,d\xi, \\
\mathcal A_2(w)&:=\int_{\halfspace}y^{a+2}|\nabla_x {U} |^2 \,dxdy=d_2\int_{\R^n}\abs{\hat w(\xi)}^2 \abs{\xi}^{2(\gamma-1)}\,d\xi, \label{notation-A2}\\
\mathcal A_3(w)&:=\int_{\halfspace} y^{a}{U}^2
\,dxdy=d_3\int_{\R^n}\abs{\hat w(\xi)}^2
\abs{\xi}^{2(\gamma-1)}\,d\xi,\label{notation-A3}
\end{align}
where
$$d_2=\frac{-a+3}{6}d_1,\quad d_3=\frac{1}{a+1}d_1.$$
\end{lemma}

\begin{proof}
We write $\mathcal A_i:=\mathcal A_i(w)$, $i=1,2,3$, for simplicity.
Note that the integrals in the right hand side of
\eqref{notation-A}, \eqref{notation-A2}, \eqref{notation-A3} are
finite because $w\in H^\gamma(\mathbb R^n)\hookrightarrow
H^{\gamma-1}(\mathbb R^n)$, and because of the definition of the Sobolev norm \eqref{Sobolev-norm}.

Thanks to \eqref{Fourier-transform} we can easily compute, using the
properties of the Fourier transform,
\begin{equation}\label{calculation-A1}
\begin{split}
\mathcal A_1:&=\int_{\R^n_+}y^{a+2}\grad {U} 2 \,dxdy=\int_{\R^n_+}y^{a+2}\lp|\nabla_x U |^2 +|\partial_y U|^2 \rp\,dxdy \\
& =\int_{\R^n}\int_0^\infty y^{a+2}\lp|\xi|^2|\hat U|^2+|\partial_y \hat U|^2\rp \,dy d\xi \\
& =\int_{\R^n}\int_0^{\infty}y^{a+2}\abs{\hat w(\xi)}^2\abs{\xi}^2 \lp\abs{\varphi(\abs{\xi}y)}^2+\abs{\varphi'(\abs{\xi}y)}^2\rp \,dyd\xi \\
& = \int_{\R^n} \abs{\hat w (\xi)}^2 \abs{\xi}^{-1-a}\int_0^\infty
t^{a+2}\lp\abs{\varphi(t)}^2+\abs{\varphi'(t)}^2\rp
\,dtd\xi \\
& =d_1\int_{\R^n}\abs{\hat w(\xi)}^2 \abs{\xi}^{-1-a} d\xi
\end{split}\end{equation}
for a constant \be\label{formula306} d_1:=\int_0^\infty
t^{a+2}\lp\abs{\varphi(t)}^2+\abs{\varphi'(t)}^2\rp  \,dt.\ee
Similarly,
\begin{equation*}
\begin{split}
\mathcal A_2:&=\int_{\R^n_+}y^{a+2}|\nabla_x U|^2 \,dxdy=\int_{\R^n}\int_0^\infty y^{a+2}|\xi|^2|\hat U|^2 \,dy d\xi \\
& =\int_{\R^n}\int_0^{\infty}y^{a+2}\abs{\hat w(\xi)}^2\abs{\xi}^2 \abs{\varphi(\abs{\xi}y)}^2 \,dyd\xi \\
& = \int_{\R^n} \abs{\hat w (\xi)}^2 \abs{\xi}^{-1-a}\int_0^\infty t^{a+2}\abs{\varphi(t)}^2  \,dtd\xi \\
& =d_2\int_{\R^n}\abs{\hat w(\xi)}^2 \abs{\xi}^{-1-a} d\xi
\end{split}\end{equation*}
for \be\label{formula308} d_2:=\int_0^\infty t^{a+2}
\abs{\varphi(t)}^2 \,dt.\ee And finally, \be\label{A3}\begin{split}
\mathcal A_3:&=\int_{\halfspace} y^a U^2  \,dxdy=
\int_{\R^n}\int_0^\infty y^{a}|\hat U|^2 \,dy d\xi =
\int_{\R^n}\int_0^\infty y^a|\hat w(\xi)|^2 |\varphi(\abs{\xi}
y)|^2\,dy d\xi
\\& =\int_{\R^n} |\hat w(\xi)|^2 \abs{\xi}^{-1-a} \int_0^\infty  t^a|\varphi(t)|^2 \,dt d\xi = d_3  \int_{\R^n} |\hat w(\xi)|^2\abs{\xi}^{-1-a}\;d\xi,
\end{split}\ee
for
$$d_3=\int_0^\infty  t^a|\varphi(t)|^2 \,dt.$$

In the next step, we find the relation  between the constants
$d_1$,$d_2$,$d_3$. All the integrals will be evaluated between zero
and infinity in the following. Multiply \eqref{equation-phi} by
$\varphi_t t^{a+3}$ and integrate by parts:
\be\label{formula300}-\int \varphi \varphi_t t^{a+3}+a\int
\varphi_t^2 t^{a+2}+\int \varphi_{tt} \varphi_t t^{a+3}=0.\ee In the
above formula, we estimate the first term by
$$\int t^{a+3}\varphi\varphi_t  =\tfrac{1}{2} \int t^{a+3} \partial_t\lp {\varphi}^2\rp =-\tfrac{a+3}{2} \int t^{a+2}\varphi^2 ,$$
and the last one by
$$\int t^{a+3}\varphi_{tt}\varphi_t =\tfrac{1}{2} \int t^{a+3}\partial_t\lp \varphi_t^2\rp =-\tfrac{a+3}{2}\int t^{a+2}\varphi_t^2 ,$$
so from \eqref{formula300} we obtain
$$(a+3)\int t^{a+2}\varphi^2 =(-a+3)\int t^{a+2}\varphi_t^2 .$$
Together with \eqref{formula306} and \eqref{formula308} this gives
$$d_1=\frac{6}{-a+3}d_2,$$
as desired.

Now, multiply equation \eqref{equation-phi} by $\varphi t^{a+2}$ and
integrate: \be\label{formula302}-\int t^{a+2}\varphi\varphi_t +a\int
t^{a+1}\varphi_t^2 +\int t^{a+2}\varphi_{tt} \varphi =0.\ee The
third term above is computed as
$$\int t^{a+2}\varphi_{tt} \varphi =-\int t^{a+2}\varphi_t^2 -(a+2)\int t^{a+1}\varphi_t \varphi , $$
so  \eqref{formula302} becomes \be\label{formula304}d_1= -2\int
t^{a+1}\varphi_t\varphi =(a+1) \int t^a\varphi^2 =(a+1) d_3.\ee This
completes the proof of the lemma.
\end{proof}

In the following, we continue the estimates of the different error
terms, although now we only need the asymptotic behavior and not the
precise constant.

\begin{lemma}\label{lemma-E}
Let $w$ be defined on $\mathbb R^n$ and $U=K_\gamma *_x w$. Then
\begin{enumerate}
\item For each $k\in \mathbb N$, if $w\in H^{\gamma-k/2}(\mathbb R^n)$,
 \be\label{integral-E-finite}\mathcal E_k:=\int_{\mathbb R^{n+1}_+} y^{a+k}\grad {U} 2\,dxdy<\infty.\ee
\item If $w\in H^{\gamma-3/2}(\mathbb R^n)$ and $(|x|w)\in H^{-1/2+\gamma}(\mathbb R^n)$, then
\be\label{integral-tilde-E-finite}\tilde{\mathcal
E}_3:=\int_{\mathbb R^{n+1}_+} y^{a}\abs{(x,y)}^3 \grad {U}
2\,dxdy<\infty. \ee
\end{enumerate}
\end{lemma}

\begin{proof}
Taking into account \eqref{Fourier-transform}, we can proceed as in
the calculation for $\mathcal A_1$ in \eqref{calculation-A1},
easily arriving at
$$\mathcal E_k=c_k\int_{\mathbb R^n} |\hat w(\xi)|^2|\xi|^{1-k-a}\,d\xi,$$
where
$$c_k:=\int_0^\infty t^{a+k}\lp \varphi^2(t)+\varphi_t^2(t)\rp \,dt<\infty,$$
and this last integral is finite for all $k\in\mathbb N$ because of
the asymptotics of the Bessel functions from Lemma
\ref{lemma-Bessel}. The second conclusion of the lemma is a little
more involved.  To show that the integral
\eqref{integral-tilde-E-finite} is finite, first note that
\eqref{integral-E-finite} with $k=3$ gives
$$\int_{\mathbb R^{n+1}_+} y^{a+3} \grad {U} 2\,dxdy<\infty.$$
It is clear that it only remains to prove
$$\int_{\mathbb R^{n+1}_+} y^{a}|x|^3 \grad {U} 2\,dxdy<\infty.$$
Since the computation of the previous integral can be made component by
component, it is clear that is enough to restrict to the case $n=1$.
Then we just need to show that \be\label{J1}J:=\int_0^\infty
\int_{\mathbb R} y^{a}|x|^3 (\partial_x {U})^2\,dxdy<\infty.\ee This
is an easy but tedious calculation using Fourier transform. Without
loss of generality, we will drop all the constants $2\pi$ appearing
in the Fourier transform. First notice that
\be\label{J}\begin{split}
\int_{\mathbb R} |x|^3(\partial_x {U})^2 \,dx &=\|\{|x|^{3/2}\partial_x U\}\|^2_{L^2(\mathbb R)}=\| D^{3/2}_\xi \widehat{\partial_x U}\|^2_{L^2(\mathbb R)}=\| D^{3/2} (\abs{\xi} \hat U)\|^2_{L^2(\mathbb R)} \\
&= \int_{\mathbb R} \abs{\xi} \hat U D_\xi^3 (\abs{\xi} \hat
U)\,d\xi. \end{split}\ee At this point we go back to
\eqref{Fourier-transform} to substitute the explicit expression for
$\hat U$. We will need to compute \bee\begin{split}
D_\xi^3 \lp|\xi|\hat w(\xi) \varphi (|\xi| y)\rp&=\hat w'''\left[|\xi|\varphi\right]+\hat w''\left[3\varphi+3|\xi|\varphi'y\right]\\
&+ \hat w'\left[6\varphi' y+3|\xi| \varphi'' y^2\right]+\hat w\left[|\xi|\varphi''' y^3+3\varphi'' y^2\right] \\
&=\hat w'''\left[|\xi|\varphi\right]+\hat w''\left[3\varphi+3t\varphi'\right]\\
&+ \hat w'\left[6|\xi|^{-1}t\varphi' +3|\xi|^{-1} t^2\varphi''
\right]+\hat w\left[|\xi|^{-2}\varphi''' t^3+3|\xi|^{-2}t^2\varphi''
\right],
\end{split}\eee
after the change $|\xi| y=t$. When we substitute the above
expression into \eqref{J} and then back into \eqref{J1}, taking into
account the change of variables, we obtain: \bee\begin{split}
J&=\int_0^\infty t^a\varphi^2 \,dt\int_{\mathbb R}\hat w'''\hat w|\xi|^{1-a}\,d\xi\\
&+ \int_0^\infty t^a \left[\varphi^2 +3t\varphi\varphi'\right]\,dt\int_{\mathbb R}\hat w''\hat w|\xi|^{-a}\,d\xi \\
& +\int_0^\infty t^a\left[6t\varphi'\varphi+3t^2\varphi''\varphi\right]dt\int_{\mathbb R} \hat w'\hat w |\xi|^{-a-1}\,d\xi \\
&+\int_0^\infty t^a \left[t^3 \varphi'''\varphi+3t^2\varphi''\varphi\right]dt\int_{\mathbb R} \hat w^2|\xi|^{-a-2}\,d\xi \\
&=:c_1 J_1+c_2 J_2+c_3 J_3+c_4 J_4.
\end{split}\eee
It is clear, looking at the asymptotic behavior of $\varphi$ from
Lemma \ref{lemma-Bessel} that the constants $c_i$, $i=1,2,3,4$, are
finite. On the other hand, by an straightforward integration by
parts argument, we can write each of the terms $J_i$, $i=1,2,3,4$,
as a linear combination of just
\be\label{two-integrals}\int_{\mathbb R} \hat w^2(\xi)
|\xi|^{-a-2}\,d\xi\quad\mbox{and}\quad \int_{\mathbb R} \hat
w'(\xi)^2 |\xi|^{-a}\,d\xi.\ee Finally, the proof is completed
because the initial hypotheses show that both integrals in
\eqref{two-integrals} are finite. In particular, these hypothesis
show that all the derivations are rigorous.
\end{proof}

\begin{lemma}\label{lemma-F}
Let $w$ be defined on $\mathbb R^n$ and $U=K_\gamma *_x w$.
\begin{enumerate}
\item For each $k\in \mathbb N$, if $w\in H^{\gamma-k/2-1} (\mathbb R^n)$,
\be\label{integral-F-finite}\mathcal F_k:=\int_{\mathbb R^{n+1}_+}
y^{a+k} U^2\,dxdy<\infty. \ee

\item If $w\in H^{\gamma-5/2}(\mathbb R^n)$ and $(|x| w)\in H^{\gamma-3/2}(\mathbb R^n)$,
\be\label{integral-tilde-F-finite}\tilde{\mathcal
F}_3:=\int_{\mathbb R^{n+1}_+} y^{a}\abs{x}^3 U^2\,dxdy<\infty. \ee
\end{enumerate}
\end{lemma}

\begin{proof}
The first assertion \eqref{integral-F-finite} follows as in
\eqref{A3}: \bee\begin{split} \mathcal F_k:&=\int_{\halfspace}
y^{a+k} U^2  \;dxdy= \int_{\R^n}\int_0^\infty y^{a+k} |\hat U|^2
\,dy d\xi = \int_{\R^n}\int_0^\infty y^{a+k} |\hat w(\xi)|^2
|\varphi(\abs{\xi} y)|^2\,dy d\xi
\\& =\int_{\R^n} |\hat w(\xi)|^2 \abs{\xi}^{-1-a-k} \int_0^\infty  |\varphi(t)|^2 t^{a+k}\,dt d\xi = c_k \int_{\R^n} |\hat w(\xi)|^2\abs{\xi}^{-1-a-k}\,d\xi,
\end{split}\eee
for
$$c_k:=\int_0^\infty  |\varphi(t)|^2 t^{a+k}\,dt<\infty.$$
For the second assertion, under the light of our previous
discussions, it is enough to show that in the one-dimensional case,
$$\int_{\mathbb R} \abs{x}^3 U^2 \,dx =\|\{|x|^{3/2}U\}\|^2_{L^2(\mathbb R)}=\| D^{3/2} \widehat{U}\|^2_{L^2(\mathbb R)}=\int_{\mathbb R} \hat U D_\xi^3 (\hat U)\,d\xi.$$
Substitute the expression for $\hat U$ from
\eqref{Fourier-transform}. Then
$$\int_{\mathbb R} \abs{x}^3 U^2 \,dx =\int \hat w'''\hat w\varphi^2\,d\xi+3\int \hat w''\hat w \varphi'\varphi y\,d\xi+3\int \hat w'\hat w  \varphi'\varphi y^2\,d\xi+\int\hat w^2\varphi''' \varphi y^3\,d\xi,$$
so when we change variables $t=\abs{\xi}y$,
\bee\begin{split}\int_0^\infty\int_{\mathbb R} y^a\abs{x}^3 U^2
\,dxdy
&=\int_0^\infty  t^{a}\varphi^2\,dt \int_{\mathbb R} \hat w'''\hat w\abs{\xi}^{-1-a}\, d\xi \\
&+3\int_0^{\infty} t^{1+a}\varphi'\varphi \,dt\int_{\mathbb R}\hat w''\hat w \abs{\xi}^{-2-a}\,d\xi \\
&+3\int_0^\infty  t^{2+a}\varphi''\varphi \,dt\int_{\mathbb R}\hat w'\hat w\abs{\xi}^{-3-a}\,d\xi\\
&+\int_0^\infty t^{3+a}\varphi''' \varphi dt\int_{\mathbb R}\hat w^2\abs{\xi}^{-4-a}\,d\xi\\
&=\tilde c_1 \tilde J_1+\tilde c_2 \tilde J_2 +\tilde c_3 \tilde J_3+\tilde c_4 \tilde J_4.
\end{split}\eee
Clearly, from the asymptotics of the Bessel functions from Lemma
\ref{lemma-Bessel}, the constants $\tilde c_i$, $i=1,2,3,4$ are
finite. At the same time, each of the four integrals $\tilde J_i$,
$i=1,2,3,4$, can be written as a linear combination of 
two:
$$\int (\hat w')^2 \abs{\xi}^{-2-a}\,d\xi\quad\mbox{and}\quad \int (\hat w)^2 \abs{\xi}^{-4-a}\,d\xi,$$
which are finite because of the hypothesis on $w$.
\end{proof}

Next, we check what happens with the previous two lemmas under
rescaling. Here $f=o(1)$ means 
$$\lim\limits_{\epsilon/\mu\to 0} f=0.$$
Given any function $w$ defined on $\mathbb R^n$, we consider its
extension to $\mathbb R^{n+1}_+$ as $U=K_\gamma *_x w$, and the
rescaling, for each $\mu>0$, \be\label{rescaling}
U_\mu(x,y):=\frac{1}{\mu^{\frac{n-2\gamma}{2}}}U\lp\frac{x}{\mu},\frac{y}{\mu}\rp.\ee

\begin{cor}\label{cor-E-mu} Fix  $\epsilon,\mu>0$ and let the hypotheses be as in Lemma \ref{lemma-E} (in each of the two cases).
\begin{enumerate}
\item For each $k\in\mathbb N$,
\be\label{estimate-E-mu} \int_{B_\epsilon^+} y^{a+k} |\nabla
U_\mu|^2dxdy=\mu^k \int_{B^+_{\epsilon/\mu}} y^{a+k} |\nabla
U|^2dxdy=\mu^k \left[\mathcal E_k+o(1)\right] \ee
\item Also
\be\label{estimate-tilde-E-mu}\int_{B_\epsilon^+} y^{a}
|(x,y)|^3\grad {U_\mu} 2\,dxdy=\mu^3\int_{B^+_{\epsilon/\mu}}
y^{a+k} \grad {U} 2\,dxdy=\mu^3\left[\tilde{\mathcal
E}_3+o(1)\right],\ee
\end{enumerate}
where $U_\mu$ is the rescaling \eqref{rescaling}, and $\mathcal E_k,
\tilde{\mathcal E}_3<\infty$ are defined as in Lemma \ref{lemma-E}.
\end{cor}

\begin{cor}\label{cor-F-mu}
Fix $\epsilon,\mu>0$ and let the hypotheses be as in Lemma
\ref{lemma-F} (in each of the two cases).
\begin{enumerate}
\item For each $k\in \mathbb N$,
\be\label{estimate-F-mu}\int_{B_\epsilon^+} y^{a+k}
(U_\mu)^2\,dxdy=\mu^{k+2}\int_{B^+_{\epsilon/\mu}} y^{a+k}
U^2\,dxdy=\mu^{k+2}\left[\mathcal F_k+o(1)\right],\ee

\item Also,
\be\label{estimate-tilde-F-mu}\int_{B_\epsilon^+} y^{a}\abs{(x,y)}^3
(U_\mu)^2\,dxdy=\mu^{5}\int_{B^+_{\epsilon/\mu}} y^{a}\abs{x}^3
U^2\,dxdy=\mu^{5}\left[\tilde{\mathcal F}_3+o(1)\right],\ee
\end{enumerate}
where $U_\mu$ is the rescaling \eqref{rescaling}, and $\mathcal F_k,
\tilde{\mathcal F}_3<\infty$ are defined as in Lemma \ref{lemma-F}.
\end{cor}

\subsection{Proof of Theorem \ref{thm2}}

We first need to choose a very particular background metric for $X$
near a non-umbilic point on $M$. We will follow the steps as Escobar
did in Lemmas 3.1 - 3.3 of \cite{Escobar:conformal-deformation}. But
our situation is a little different. Our freedom of choice
of metrics is restricted to the boundary. Hence we will make some
assumptions on the behavior of the asymptotically hyperbolic
manifolds in order to allow us to see clearly what we can get for a
good choice of representative from the conformal infinity.

\begin{lemma}\label{curvature-condition} Suppose that $(X^{n+1}, \ g^+)$
is an asymptotically hyperbolic manifold and $\rho$ is a geodesic
defining function associated with a representative $\hat h$ of the
conformal infinity $(M^n, [\hat h])$. Assume that \be
\label{curv-cond}  \rho^{-2}\big(R[g^+] - Ric[g^+](\rho\partial_\rho) +
n^2\big) \to 0 \quad \text{as $\rho\to 0$.}\ee  Then, at $\rho=0$,
\be\label{mean-curv-v} H := \text{Tr}_{\hat h} h^{(1)} = 0\ee and
\be \label{normal-ricci} \text{Tr}_{\hat h}h^{(2)} = \frac
12(\|h^{(1)}\|^2_{\hat h} + \frac 1{2(n-1)}R[\hat h]),\ee where
$$
g^+ = \frac {d\rho^2 + h_\rho}{\rho^2}, \quad h_\rho = \hat h +
h^{(1)}\rho + h^{(2)}\rho^2 + o(\rho^2).
$$
\end{lemma}

\begin{proof} This simply follows from the calculations in
\cite{Graham}. Recall (2.5) from \cite{Graham} \be\label{cal-graham}
\begin{split} \rho h_{ij}'' + (1-n)h_{ij}' & - h^{kl}h_{kl}'h_{ij} - \rho h^{kl}h_{ik}'h_{jl}' +
\frac 12 \rho h^{kl}h_{kl}' h_{ij}'- 2\rho R_{ij}[\hat h] \\ &  =
\rho (R_{ij}[g^+] + n g^+_{ij}), \end{split}\ee where we use $h$ to
stand for $h_\rho$ for simplicity. Taking its trace with respect to
the metrics $h$, we have \be\label{cal-graham-trace}
\begin{split} \rho\text{Tr}_h h'' &  + (1-2n)\text{Tr}_{h}h'- \rho\|h'\|_h^2 +
\frac 12 \rho (\text{Tr}_h h')^2- 2\rho R[\hat h]  \\ & = \rho^{-1}
(R[g^+] - Ric[g^+](x\partial_x) + n^2)\end{split}\ee Immediately
from \eqref{curv-cond} we see that
$$
\text{Tr}_h h' = 0 \quad\text{at $\rho=0$}.
$$
Then, dividing $\rho$ in both sides of the equation
\eqref{cal-graham-trace} and taking $\rho\to 0$, we have
\eqref{normal-ricci}, under the assumption \eqref{curv-cond},
because
$$
(\text{Tr}_h h')' = \text{Tr}_{\hat h} h'' - \|h'\|^2_{\hat h}
$$
at $\rho=0$.
\end{proof}

Notice that \eqref{curv-cond} is an intrinsic curvature condition of
an asymptotically hyperbolic manifold, which is independent of the
choice of geodesic defining functions. Consequently we have the following.

\begin{lemma}\label{lemma-structure-metric2} Suppose that $(X^{n+1},
\ g^+)$ is an asymptotically hyperbolic manifold and
\eqref{curv-cond} holds. Then, given a point $p$ on the boundary
$M$, there exists a representative $\hat h$ of the conformal
infinity such that,
\begin{itemize}
\item[i.] $H = :\text{Tr}_{\hat h}h^{(1)} =  0$ on $M$,
\item[ii.] $Ric[\hat h](p)=0$ on $M$,
\item[iii] $Ric [\bar g] (\partial_\rho)(p)=0$ on $M$,
\item[iv.] $R[\bar g](p)= \|h^{(1)}\|_{\hat h}^2$ on $M$.
\end{itemize}
\end{lemma}

\begin{proof} The proof, like the proof of Lemma 3.3 in
\cite{Escobar:conformal-deformation}, uses Theorem 5.2 in
\cite{Lee-Parker}. Therefore we may choose a representative of the
conformal infinity whose Ricci curvature vanishes at any given point
$p\in M$. In the light of Lemma \ref{curvature-condition}
we get \emph{i.} and \emph{ii.} right away. We then calculate
$$
Ric[\bar g](\partial_x)= - \frac 12 \text{Tr}_{\hat h}h^{(2)} +
\frac 14 \|h^{(1)}\|_{\hat h}^2 = 0
$$
at $p\in M$ from \eqref{normal-ricci}. Finally we recall that
$$
R[\bar g] = 2Ric[\bar g](\partial_\rho) + R[\hat h] +
\|h^{(1)}\|_{\hat h}^2 - (\text{Tr}_{\hat h}h^{(1)})^2 =
\|h^{(1)}\|_{\hat h}^2.
$$
The proof is complete.
\end{proof}

Assume that $0\in M=\partial \bar X$ is a non-umbilic point. Choose
normal coordinates $x_1,\ldots,x_n$ around $0$ on $M$ and let
$(x_1,\ldots,x_n,\rho)$ be the Fermi coordinates on $X$ around
$0$. In particular, we can write
$$
g^+=\rho^{-2}(d\rho^2+h_{ij}(x,\rho)dx_i dx_j),\quad \bar g =
d\rho^2+h_{ij}(x,\rho)dx_i dx_j.
$$
In order to simplify the later notation, we denote the coordinate
$\rho$ by $y$. The only risk of confusion comes from the fact that
we have previously used $y$ for the special defining function
$\rho^*$ from Proposition \ref{new-defining-function}, but we will
not need it any longer. In the new notation we have
$$
\bar g=dy^2+h_{ij}(x,y)dx_i dx_j
$$
for some functions $h_{ij}(x,y)$, $i,j=1,\ldots,n$. From what we
have in the above two lemmas we get from Lemma 3.1 and 3.2 of
\cite{Escobar:conformal-deformation} the following.

\begin{lemma}\label{lemma-structure-metric1} Suppose that $(X^{n+1}, \ g^+)$
is an asymptotically hyperbolic manifold satisfying
\eqref{curv-cond}. Given a non-umbilic point $p$ on the boundary
$M$, i.e. $\|h^{(1)}\|_{\hat h}(p)\neq 0$ for $p\in M$, where $\hat
h$ is chosen as in Lemma \ref{lemma-structure-metric2}. Then:
\begin{enumerate}
\item $\sqrt{|\bar g|}=1-\tfrac{1}{2}\norm{\pi}^2 y^2+O(\abs{(x,y)}^3)$.
\item $\bar g^{ij}=\delta_{ij} + 2 \pi^{ij} y- \tfrac{1}{3}
R^{i \quad j}_{ \ k l} [\hat h] \, x_k x_l + {{\bar g}^{ij}}\!_{,ym}
yx_m +\lp 3 \pi^{im}{\pi_m}^j+{{R^i}_y\!^j\!_y}[\bar g]\rp
y^2+O(\abs{(x,y)}^3)$,
\end{enumerate}
where, for simplicity, we set $\pi= h^{(1)}$.
\end{lemma}

As in Proposition \ref{prop-compare-sphere}, we try to find a good
test function for the Sobolev quotient given by \bee
I^*_\gamma[U,\bar g]=\frac{d_\gamma^*\int_X y^a\abs{\nabla U}_{\bar
g}^2 \,dv_{\bar g} + \int_X E(y) U^2 \,dv_{\bar g}}{\lp \int_M
|U|^{2^*} \,dv_{\hat h}\rp^{\frac{2}{2^*}}},\eee where $E(y)$ is
given by \eqref{E2}, with respect to the metric $\bar g$:
\be\label{Error}E(y)=\frac{n-1-a}{4n}\left[R[{\bar
g}]-(n(n+1)+R[{g^+}])y^{-2}\right] y^a.\ee We need to perform a
careful computation of the lower order terms in order to find an
estimate for $\Lambda_\gamma(M,[\hat h])$. For simplicity, we
introduce the following notation: for a subset $\Omega\subset
\mathbb R^{n+1}_+$, we consider the energy functional restricted to
$\Omega$ given by
$$\mathcal K(U,\Omega):=d_\gamma^*\int_\Omega
y^a\abs{\nabla U}_{\bar g}^2 \,dv_{\bar g} + \int_\Omega E(y) U^2
\,dv_{\bar g}$$

Given any $\epsilon>0$, let $B_\epsilon$ be the ball of radius
$\epsilon$ centered at the origin in $\mathbb R^{n+1}$ and
$B_\epsilon^+$ be the half ball of radius $\epsilon$ in $\mathbb
R^{n+1}_+$. Choose a smooth radial cutoff function $\eta$, $0\leq
\eta\leq 1$, supported on $B_{2\epsilon}$, and satisfying $\eta=1$
on $B_\epsilon$. We recall here the conformal diffeomorphisms of the
sphere $w_\mu$ given in \eqref{diffeo-sphere} and their extension
$U_\mu$ as in \eqref{diffeo-sphere-extension}. Our test function is
simply
$$V_\mu:=\eta U_\mu.$$

\noindent\emph{Step 1: Computation of the energy in $B_\epsilon^+$.}
\vskip 0.1in

It is clear that in the half ball $B_\epsilon^+$, $V_\mu=U_\mu$, so
that $\mathcal K (V_\mu,B_\epsilon^+)=\mathcal K
(U_\mu,B_\epsilon^+)$. We compute the first term in the energy
$\mathcal K (U_\mu,B_\epsilon^+)$.  Using the asymptotics for $\bar
g$ from Lemma \ref{lemma-structure-metric1} (here the indexes $i,j$
run from $1$ to $n$),
\be\label{big-formula}
\begin{split}
\int_{B_\epsilon^+} & y^a\abs{\nabla U_\mu}^2_{\bar g}\, dv_{\bar g}
= \int_{B_\epsilon^+} y^a\left[\bar g^{ij}\lp\partial_i U_\mu
\rp\lp\partial_j U_\mu \rp+(\partial_y U_\mu)^2\right]\,dv_{\bar g} \\
& = \int_{B_\epsilon^+} y^a\abs{\nabla U_\mu}^2\,dv_{\bar
g}\\
&+2\pi^{ij}\int_{B_\epsilon^+} y^{a+1}\lp\partial_i U_\mu \rp\lp\partial_j U_\mu \rp\,dv_{\bar g} \\
& + \int_{B_\epsilon^+}y^{a+2}\lp 3\pi^{im}
{\pi_m}^j+{{R^i}_y\!^j\!_y}[\bar g]\rp
\lp\partial_i U_\mu \rp\lp\partial_j U_\mu \rp\,dv_{\bar g}\\
&+\int_{B_\epsilon^+} y^{a+1}{\bar g^{ij}}\!_{,tk}  x_k \lp
\partial_i U_\mu \rp\lp\partial_j U_\mu \rp\,dv_{\bar g} \\
&-\tfrac{1}{3}\int_{B_\epsilon^+} y^a {R^i\!_{kl}}\!^j[\bar g] x_k
x_l
\lp\partial_i U_\mu \rp\lp\partial_j U_\mu \rp\,dv_{\bar g} \\
& + c\int_{B_\epsilon^+} y^a\abs{(x,y)}^3\abs{\nabla U_\mu}^2\,dv_{\bar g}\\
& =: J_1+J_2+J_3 +J_4 +J_5+J_6.
\end{split}\ee

We estimate the first integral $J_1$ in the right hand side of
\eqref{big-formula}, using the estimate for the volume element
$\sqrt{|\bar g|}$ from Lemma \ref{lemma-structure-metric1}:
\be\label{formula204}\begin{split}
J_1&=\int_{B_\epsilon^+} y^a\abs{\nabla U_\mu}^2\,dv_{\bar g} \\
&\leq \int_{B_\epsilon^+} y^a\abs{\nabla U_\mu}^2\,dxdy-\tfrac{1}{2}\norm{\pi}^2\int_{B_\epsilon^+} y^{2+a}\abs{\nabla U_\mu}^2\,dxdy\\
&+c\int_{B_\epsilon^+}y^a\abs{\nabla U_\mu}^2\abs{(x,y)}^3\,dxdy \\
& \leq \int_{B_\epsilon^+} y^a\abs{\nabla U_\mu}^2\,dxdy
-\tfrac{1}{2}\norm{\pi}^2\mu^2\mathcal A_1+\mu^2 o(1)+c\mu^3
\left[\tilde{\mathcal E}_3+o(1)\right],
\end{split}\ee
if we take into account the notation from \eqref{notation-A} and
Corollary \ref{cor-E-mu}.

Now we look closely at the equation for $U_\mu$. Multiply expression
\eqref{extension-Rn} by $U_\mu$ and integrate by parts:
\be\label{formula100}\int_{B_\epsilon^+} y^a\abs{\nabla
U_\mu}^2\,dxdy=c_{n,\gamma}\int_{\Gamma_\epsilon^0}
w_\mu^{2^*}\,dx+\int_{\Gamma_\epsilon^+} U_\mu \lp\partial_\nu
U_\mu\rp\,d\sigma\leq c_{n,\gamma} \int_{\Gamma_\epsilon^0}
w_\mu^{2^*}\,dx,\ee where $\nu$ is the exterior normal to
$B_\epsilon^+$. Here we have used the properties of the convolution
with a radially symmetric, nonincreasing kernel $K_\gamma$. More precisely, since
$w_\mu$ is radially symmetric and non-increasing, 
$U_\mu=K_\gamma *_x w_\mu$ also satisfies $\partial_\nu U_\mu\leq 0$ on
$\Gamma_\epsilon^+$ (c.f.
\cite{Cabre-Roquejoffre:front-propagation}, Lemma 2.3, for
instance).

From \eqref{formula100}, using \eqref{formula202}, we arrive at
\be\label{formula208}\int_{B_\epsilon^+} y^a\abs{\nabla
U_\mu}^2\,dxdy\leq
\Lambda(S^m,[g_c])(d_\gamma^*)^{-1}\left[\int_{\Gamma_\epsilon^0}
(w_\mu)^{2^*}\,dx\right]^{\frac{n-2\gamma}{n}}.\ee For simplicity,
we set $\Lambda_1:=\Lambda(S^m,[g_c])(d_\gamma^*)^{-1}$. Equations
\eqref{formula204} and \eqref{formula208} tell us that
\be\label{formula206}\begin{split} J_1=\int_{B_\epsilon^+}
y^a\abs{\nabla U_\mu}^2\,dv_{\bar g}\leq
\Lambda_1\left[\int_{\Gamma_\epsilon^0}
(w_\mu)^{2^*}\,dx\right]^{\frac{2}{2^*}}-\tfrac{1}{2}\norm{\pi}^2
\mu^2 \mathcal A_1 +\mu^2 o(1)+ c\mu^3  .
\end{split}\ee

On the other hand, the asymptotics for the metric $\hat h =\bar
g|_{y=0}$ near the origin are explicit. Indeed, from Lemma
\ref{lemma-structure-metric2} we know that
\be\label{relate-metrics}\sqrt{|\hat h|}=1+O(\abs{x}^3).\ee
Moreover, we can compute from \eqref{scaling-properties}
$$\int_{\Gamma_\epsilon^0} (w_\mu)^{2^*}\abs{x}^3\, dx=\mu^3\int_{\Gamma_{\epsilon/\mu}^0} (w_1)^{2^*} \abs{x}^3\,dx\leq c\mu^3.$$
Consequently, from \eqref{relate-metrics} we are able to relate the
integrals in $dv_{\hat h}$ and $dx$:
$$\int_{\Gamma_\epsilon^0} (w_\mu)^{2^*}\,dx\leq \int_{\Gamma_\epsilon^0} (w_\mu)^{2^*}\,dv_{\hat h}+c\mu^3.$$
And substituting the above expression into \eqref{formula206} we get
\bee J_1=\int_{B_\epsilon^+} y^a\abs{\nabla U_\mu}^2\,dv_{\bar
g}\leq \Lambda_1\left[\int_{\Gamma_\epsilon^0}
(w_\mu)^{2^*}\,dv_{\hat h}\right]^{\frac{2}{2^*}}
-\tfrac{1}{2}\norm{\pi}^2 \mu^2 \mathcal A_1 +\mu^2 o(1)+c\mu^3.
\eee

Now we go back to \eqref{big-formula}, and try to estimate the
second term $J_2$ in the right hand side. If we again use  the
asymptotics of the metric $\bar g$ given in Lemma
\ref{lemma-structure-metric1}, then
\be\label{formula70}\begin{split} \int_{B_\epsilon^+} y^{a+1} &\lp
\partial_i U_\mu\rp \lp\partial_j U_\mu\rp\,dv_{\bar g} \leq
\int_{B_\epsilon^+} y^{a+1} \lp \partial_i U_\mu\rp \lp\partial_j
U_\mu\rp\,dxdy +\mathcal B,
\end{split}\ee
for \bee \mathcal B\leq c\int_{B_\epsilon^+} y^{a+3} \abs{\nabla
U_\mu}^2 dxdy+c\int_{B_\epsilon^+} y^{a+1}|\nabla {U_\mu}|^2
\abs{(x,y)}^3\,dxdy. \eee We notice here that $\mathcal B$ can be
easily estimated from  Corollary \ref{cor-E-mu}:
\be\label{B}\mathcal B\leq c \mu^3(\mathcal E_3+o(1))+c\mu^3\epsilon
\lp \tilde{\mathcal E}_3+o(1)\rp\leq c \mu^3+\mu^3 o(1).\ee

Let us look at the cross terms $(\partial_i U_\mu)(\partial_j
U_\mu)$, $1\leq i,j\leq n$ in \eqref{formula70}. We note that
$\partial_i U_\mu=K_\gamma *_x (\partial_i w_\mu)$, just by  taking
the derivatives in the convolution. This last derivative can be
explicitly written, and in particular, $\partial_i w_\mu$ is an odd
function in the variable $x_i$. By the properties of the
convolution, we know that $\partial_i U_\mu$ is also an odd function
in the variable $x_i$. Then, using the symmetries of the half ball,
the integral $\int_{B_\epsilon^+} y^{a+1}(\partial_i
U_\mu)(\partial_j U_\mu)\,dxdy$ is zero if $i\neq j$.  If $i=j$, we
 use that the mean curvature at the point vanishes,
i.e., $\pi_i^{i}=0$ by Lemma \ref{lemma-structure-metric2}. Then,
when we substitute formula \eqref{formula70} in the expression for
$J_2$, only the error term remains, and  by \eqref{B} we conclude
that \be\label{J2}J_2=2 \pi^{ij}\int_{B_\epsilon^+} y^{a+1} \lp
\partial_i U_\mu\rp \lp\partial_j U_\mu\rp\,dv_{\bar g}\leq
\mathcal B\leq  \mu^3(c+o(1)).\ee

Now we estimate the next term in \eqref{big-formula}, $J_3$. Again using
the asymptotics for the volume element $dv_{\bar g}$ from
Lemma \ref{lemma-structure-metric1}, we have that
\be\label{formula71} \int_{B_\epsilon^+} y^{a+2} \lp\partial_i
U_\mu\rp\lp\partial_j U_\mu\rp\; dv_{\bar g} \leq
\int_{B_\epsilon^+} y^{a+2} \lp\partial_i U_\mu\rp\lp\partial_j
U_\mu\rp\;dxdy+\mathcal B', \ee for \bee\begin{split}
\mathcal B'&\leq  c\int_{B_\epsilon^+} y^{a+4}\abs{\nabla U_\mu}^2dxdy
+c\int_{B_\epsilon^+}y^{a+2}\abs{(x,y)}^3 \abs{\nabla U_\mu}^2 dxdy \\
& \leq \mu^4 (\mathcal E_4+o(1))+\mu^3\epsilon^2 (\tilde{\mathcal
E}_3+o(1))\leq c\mu^3,
\end{split}\eee
where the last estimate follows thanks to Corollary \ref{cor-E-mu}
again.

Notice again that, for $i\neq j$ the first integral in the right
hand side of \eqref{formula71} vanishes - thanks to the symmetries
of the half ball and the discussion above on the oddness of the
derivatives of $U_\mu$. Then, we recall the definition of $\mathcal
A_2$ from \eqref{notation-A2} and the estimate \eqref{estimate-E-mu}.
When we put all these ingredients together:
\bee\begin{split}J_3&=\lp  3\pi^{im} {\pi_m}^j+{{R^i}_y\!^j\!_y}[\bar g]
\rp\int_{B_\epsilon^+}y^{a+2}\lp\partial_i U_\mu \rp\lp\partial_j U_\mu \rp\;dv_{\bar g} \\
& = \frac{1}{n}\left[3 \norm{\pi}^2+Ric(\nu)\right] \mu^2 \mathcal A_2+c\mu^3\\
& = \frac{3}{n} \norm{\pi}^2 \mu^2 \mathcal A_2+\mu^2 o(1)+c\mu^3.
\end{split}\eee
if we take into account that $Ric(\nu)(0)[\hat h]=0$ because of
Lemma \ref{lemma-structure-metric2}.

Next, the calculation for $J_4$ is very similar to the previous one.
Indeed,
$$\int_{B_\epsilon^+} y^{a+1}x_k\lp\partial_i U_\mu \rp\lp\partial_j U_\mu \rp \;dv_{\bar g}\leq
\int_{B_\epsilon^+} y^{a+1} x_k\lp\partial_i U_\mu \rp\lp\partial_j U_\mu \rp\;dxdy + \mathcal B'',$$
and because of symmetries on the unit ball, the first integral in
the right hand side above vanishes for all $i,j,k$, while $\mathcal
B''\leq c \mu^3$. Thus
$$
J_4={\bar g^{ij}}\!_{,tk}\int_{B_\epsilon^+} y^{a+1}  x_k
\lp\partial_i V_\mu \rp\lp\partial_j V_\mu \rp \;dv_{\bar g}\leq
c\mu^3.
$$
And finally $J_5$, $J_6$ can be estimated in a similar
manner.

Putting all the estimates together for the $J_j$, $j=1,\ldots,6$, we
have shown that \eqref{big-formula} reduces to
\be\label{energy1}\int_{B_\epsilon^+} y^a\abs{\nabla U_\mu}_{\bar
g}^2\,dv_{\bar g}\leq \Lambda_1\left[\int_{\Gamma_\epsilon^0}
(w_\mu)^{2^*}\,dv_{\hat h}\right]^{\frac{2}{2^*}}+\left[
-\tfrac{1}{2}\mathcal A_1+\tfrac{3}{n}\mathcal
A_2\right]\norm{\pi}^2 \mu^2 + \mu^2 o(1)+c\mu^3.\ee

Finally, we are able to complete the computation of the energy
$\mathcal K(U_\mu,B_\epsilon^+)$. Note that in the half ball
$B_\epsilon^+$, we have a very precise behavior for the lower order
term \eqref{Error}. In particular, Lemma \ref{lemma-structure-metric2} gives that $R[{\bar g}](p)=\norm{\pi}^2$, so
\be E(y)=\frac{n-1+a}{4n}\norm{\pi}^2 y^a+O(y^{1+a}).\ee
Then, again using  the asymptotics for the volume element $dv_{\bar
g}$, \be\label{formula80} \int_{B_\epsilon^+} E(y)(U_\mu)^2
\,dv_{\bar g}= \frac{n-1+a}{4n} \norm{\pi}^2\int_{B_\epsilon^+}
y^a(U_\mu)^2\,dxdy +\mathcal B''',\ee
where \bee \mathcal B'''\leq
c\int_{B_\epsilon^+} y^{a+1} (U_\mu)^2\,dxdy+c\int_{B_\epsilon^+}
y^a\abs{x}^3 (U_\mu)^3\,dxdy \eee can be estimated from Corollary
\ref{cor-F-mu} as \be\label{formula81}\mathcal B'''\leq c
\mu^3+o(1).\ee Summarizing, from \eqref{formula80} and
\eqref{formula81}, and using the scaling properties of $U_\mu$ as
given in \eqref{scaling-properties}, we have
\be\label{energy2}\begin{split} \int_{B_\epsilon^+} E(y)(U_\mu)^2
\;dv_{\bar g} &\leq\tfrac{n-1+a}{4n}
\norm{\pi}^2 \mu^{2} \int_{B^+_{\epsilon/\mu}} y^a(U_1)^2\;dxdy +c\mu^3 \\
& = \tfrac{n-1+a}{4n} \norm{\pi}^2 \mu^{2} \mathcal
A_3+c\mu^2o(1)+c\mu^3,
\end{split}\ee
where for the last inequality we have used Corollary \ref{cor-F-mu}
and the definition of $\mathcal A_3$ from \eqref{notation-A3}.

The energy of $V_\mu$ in the half ball $B_\epsilon^+$ is computed
from \eqref{energy1} and \eqref{energy2}, noting that
$\Lambda_1=\Lambda(S^n,[g_c])d_\gamma^*$, and the relation between
$\mathcal A_1$,$\mathcal A_2$,$\mathcal A_3$ from Lemma
\ref{lemma-compare-integrals}: \bee\begin{split}&\mathcal K
(V_\mu,B_\epsilon^+)=d_\gamma^*\int_{B_\epsilon^+}
y^a\abs{\nabla U_\mu}^2 dv_{\bar g} + \int_{B_\epsilon^+} E(y) (U_\mu)^2 \,dv_{\bar g}\\
&\leq\Lambda(S^n,[g_c])\left[\int_{\Gamma_\epsilon^0}
(w_\mu)^{2^*}dv_{\hat h}\right]^{\frac{2}{2^*}}+\left[
d_\gamma^*\lp-\tfrac{1}{2}
\mathcal A_1+\tfrac{3}{n}\mathcal A_2\rp+\tfrac{n-1+a}{4n}\mathcal A_3\right]\norm{\pi}^2 \mu^2  + \mu^2 o(1)+ c\mu^3 \\
&\leq \Lambda(S^n,[g_c])\left[\int_{\Gamma_\epsilon^0}
(w_\mu)^{2^*}\,dv_{\hat
h}\right]^{\frac{2}{2^*}}+\theta_{n,\gamma}\norm{\pi}^2 \mu^2
\int_{\mathbb R^n} \abs{\xi}^{2(\gamma-1)}|\hat w_1(\xi)|^2 \,d\xi + \mu^2
o(1)+ c\mu^3
\end{split}\eee
for \be\label{theta-n}
\theta_{n,\gamma}=\frac{1}{4n}\left[-\frac{n+a-3}{1-a}2^{2\gamma+1}
\frac{\Gamma(\gamma)}{\Gamma(-\gamma)}+\frac{n-1+a}{a+1}\right]d_1.
\ee
Finally, we note that the $w_1\in H^{\gamma}(\mathbb R^n)$ and $(|x|w)\in
H^\gamma(\mathbb R^n)$, so that all our computations are well justified.

\vskip 0.1in \noindent\emph{Step 2: Computation of the energy in the half-annulus
$B_{2\epsilon}^+ \backslash B_\epsilon^+$.}

\vskip 0.1in In order to compute $\mathcal K (V_\mu,B_{2\epsilon}^+\backslash
B_\epsilon^+)$, note that
$$|\nabla{V_\mu}|^2_{\bar g}\leq c\grad {V_\mu} 2 \leq c\lp\eta^2\grad {U_\mu} 2+(U_\mu) ^2 \grad \eta 2\rp$$
so that, because of the structure of the cutoff function $\eta$,
\be\label{formula90}|\nabla{V_\mu}|^2_{\bar g}\leq c \grad {U_\mu} 2
+ \frac{c}{\epsilon}(U_\mu)^2.\ee
Moreover,
\be\label{formula91}\int_{B_{2\epsilon}^+\backslash B_\epsilon^+}
y^a\lp U_\mu\rp^2 \,dxdy\leq \mu^2
\int_{B_{2\epsilon/\mu}^+\backslash B_{\epsilon/\mu}^+} y^a
(U_1)^2\,dxdy= \mu^2 o(1),\ee because the integral $\int_{\mathbb
R^n} y^a (U_1)^2\,dxdy$ is finite and $\epsilon/\mu\to \infty$. On
the other hand, we know that
$$ \lp\frac{\epsilon}{\mu}\rp^{3}\int_{B_{2\epsilon/\mu}^+\backslash
B_{\epsilon/\mu}^+} y^a |\nabla U_1|^2\,dxdy \leq \int_{B_{2\epsilon/\mu}^+\backslash
B_{\epsilon/\mu}^+} y^a \abs{(x,y)}^3|\nabla U_1|^2\,dxdy\leq \tilde{ \mathcal E}_3<\infty$$
because of Lemma \ref{lemma-F}. As a consequence,
\be\label{formula92}\int_{B_{2\epsilon}^+\backslash B_\epsilon^+}
y^a |\nabla U_\mu|^2\,dxdy=\int_{B_{2\epsilon/\mu}^+\backslash
B_{\epsilon/\mu}^+} y^a |\nabla U_1|^2\,dxdy\leq
\lp\frac{\mu}{\epsilon}\rp^3 \tilde{\mathcal E}_3.\ee If we put
together formulas \eqref{formula90}, \eqref{formula91} and
\eqref{formula92} we arrive at
$$\mathcal K (V_\mu,B_{2\epsilon}^+\backslash B_\epsilon^+)
=\int_{B_{2\epsilon}^+\backslash B_\epsilon^+} y^a |\nabla U_\mu|^2\,dxdy+
\int_{B_{2\epsilon}^+\backslash B_\epsilon^+} E(y)(U_\mu)^2\,dxdy\leq \mu^2 o(1)$$
when $\mu/\epsilon\to 0$.\\

\noindent\emph{Step 3: Completion of the proof.} \vskip 0.1in

We have very carefully computed  \bee\begin{split}\mathcal K &(V_\mu,
X)=d_\gamma^*\int_{X}y^a\abs{\nabla V_\mu}^2 \dvol_{\bar g} + \int_X
E(y) (V_\mu)^2
\dvol_{\bar g}\\
&\leq \Lambda(S^n,[g_c])\left[\int_{\Gamma_\epsilon^0}
(w_\mu)^{2^*}\,dv_{\hat h}\right]^{\frac{2}{2^*}}+
\theta_{n,\gamma}\norm{\pi}^2 \mu^2 \int_{\mathbb R^n} |\hat
w_1(\xi)|^2 |\xi|^{2(\gamma-1)}\,d\xi  + \mu^2 o(1)+ c\mu^3,
\end{split}\eee
where $\theta_{n, \gamma}$ is given in \eqref{theta-n}.

If there is a non-umbilic point, $\norm{\pi}^2\neq 0$ at that
point. In the case that $\theta_{n,\gamma}<0$, we are done,
because fixing $\epsilon$ small and then choosing $\mu$ much
smaller, then \bee\begin{split}\mathcal K (V_\mu, X)<
\Lambda(S^n,[g_c])\left[\int_{M} (w_\mu)^{2^*}\,dv_{\hat
h}\right]^{\frac{2}{2^*}},
\end{split}\eee
as desired. \qed

\bigskip

\noindent\textbf{Acknowledgements:}

The authors would like to acknowledge the hospitality of USTC
(Hefei, China), UC Santa Cruz (USA) and the Beijing International
Center for Mathematical Research (China).

\bibliographystyle{abbrv}

\end{document}